\documentclass[11pt]{article}
\usepackage{amssymb,amsmath,amsfonts,amsthm,mathrsfs}
\usepackage{graphicx,graphics,psfrag,epsfig,calrsfs,cancel}
\usepackage[colorlinks=true, pdfstartview=FitV, linkcolor=blue, citecolor=red, urlcolor=blue]{hyperref}

\usepackage{caption,subfig,color}
\usepackage{tikz}
\usetikzlibrary{calc}

\usepackage{geometry}
\newtheorem{assumption}{Assumption}

\newtheorem{lemma}{Lemma}
\newtheorem{proposition}{Proposition}
\newtheorem{theorem}{Theorem}
\newtheorem{corollary}{Corollary}
\newtheorem{definition}{Definition}

\newcommand{\N}{{\mathbb N}}
\newcommand{\Z}{{\mathbb Z}}
\newcommand{\R}{{\mathbb R}}
\newcommand{\C}{{\mathbb C}}

\newcommand{\E}{{\mathbb E}}

\newcommand{\M}{{\mathbb M}}
\newcommand{\D}{{\mathbb D}}
\newcommand{\Dbar}{\overline{\mathbb D}}
\newcommand{\U}{{\mathcal U}}
\newcommand{\Ubar}{\overline{\mathcal U}}
\newcommand{\cercle}{{\mathbb S}^1}

\newcommand{\G}{{\mathcal G}}
\newcommand{\T}{{\mathcal T}}
\newcommand{\md}{\mathrm{d}}
\newcommand{\bqs}{\begin{equation*}}
\newcommand{\eqs}{\end{equation*}}
\newcommand{\bqq}{\begin{equation}}
\newcommand{\eqq}{\end{equation}}
\renewcommand{\Re}{\mathrm{Re}}
\renewcommand{\Im}{\mathrm{Im}}

\begin{document}

\title{Sharp stability for finite difference approximations of \\
hyperbolic equations with boundary conditions}

\author{Jean-Fran\c{c}ois {\sc Coulombel} \& Gr\'egory {\sc Faye}\thanks{Institut de Math\'ematiques de Toulouse - UMR 5219, Universit\'e de 
Toulouse ; CNRS, Universit\'e Paul Sabatier, 118 route de Narbonne, 31062 Toulouse Cedex 9 , France. Research of J.-F. C. was supported 
by ANR project Nabuco, ANR-17-CE40-0025. G.F. acknowledges support from an ANITI (Artificial and Natural Intelligence Toulouse Institute) 
Research Chair and from Labex CIMI under grant agreement ANR-11-LABX-0040. Emails: {\tt jean-francois.coulombel@math.univ-toulouse.fr}, 
{\tt gregory.faye@math.univ-toulouse.fr}}}
\date{\today}
\maketitle

\begin{abstract}
In this article, we consider a class of finite rank perturbations of Toeplitz operators that have simple eigenvalues on the unit circle. Under a suitable 
assumption on the behavior of the essential spectrum, we show that such operators are power bounded. The problem originates in the approximation 
of hyperbolic partial differential equations with boundary conditions by means of finite difference schemes. Our result gives a positive answer to a 
conjecture by Trefethen, Kreiss and Wu that only a weak form of the so-called Uniform Kreiss-Lopatinskii Condition is sufficient to imply power 
boundedness.
\end{abstract}
\bigskip


\bigskip
\bigskip


Throughout this article, we use the notation
\begin{align*}
&\U := \{\zeta \in \C,|\zeta|>1 \}\, ,\quad \D := \{\zeta \in \C,|\zeta|<1 \}\, ,\quad \cercle := \{\zeta \in \C,|\zeta|=1 \} \, ,\\
&\Ubar := \U \cup \cercle \, ,\quad \Dbar := \D \cup \cercle \, .
\end{align*}
If $w$ is a complex number, the notation $B_r(w)$ stands for the open ball in $\C$ centered at $w$ and with radius $r>0$, that is $B_r(w) := 
\{ z \in \C \, / \, |z-w|<r \}$. We let ${\mathcal M}_{n,k} (\C)$ denote the set of $n \times k$ matrices with complex entries. If $n=k$, we simply 
write ${\mathcal M}_n (\C)$.

Eventually, we let $C$, resp. $c$, denote some (large, resp. small) positive constants that may vary throughout the text (sometimes within the same line). 
The dependance of the constants on the various involved parameters is made precise throughout the article.

\section{Introduction}

This article is devoted to the proof of power boundedness for a class of finite rank perturbations of some Toeplitz operators. The problem originates 
in the discretization of initial boundary value problems for hyperbolic partial differential equations. From the standard approach in numerical analysis, 
convergence of numerical schemes follows from stability and consistency. We focus here on stability. For discretized hyperbolic problems with numerical 
boundary conditions, several possible definitions of stability have been explored. From a historic perspective, the first stability definition introduced for 
instance in \cite{kreiss1,osher1,osher2} is a power boundedness property and reads (here $T$ denotes the discrete evolution operator which gives 
the solution at each time step, and the norm in \eqref{powerbound} below corresponds to an operator norm on $\ell^2(\N)$ - the numerical boundary 
conditions are incorporated in the definition of the functional space):
\begin{equation}
\label{powerbound}
\sup_{n \in \N} \, \| \, T^{\, n} \, \| \, < \, + \, \infty \, .
\end{equation}
The notion of \emph{strong stability} later introduced in the fundamental contribution \cite{gks} amounts to proving a strengthened version of the resolvent 
condition :
\begin{equation}
\label{resolvent}
\sup_{z \in \U} \, (|z| \, - \, 1) \, \| \, (z\, I  \, - \, T)^{-1} \, \| \, < \, + \, \infty \, .
\end{equation}
We refer to \cite{SW} for a detailed exposition of the links between the conditions \eqref{powerbound} and \eqref{resolvent}. Both conditions 
\eqref{powerbound} and \eqref{resolvent} preclude the existence of unstable eigenvalues $z \in \U$ for the operator $T$, the so-called 
Godunov-Ryabenkii condition \cite{gko}.

The notion of strong stability analyzed in \cite{gks} has the major advantage of being stable with respect to perturbations. It is an \emph{open} condition, 
hence suitable for nonlinear analysis. However, it is restricted to zero initial data and is therefore not so convenient in practical applications. A long line 
of research has dealt with proving that strong stability implies power boundedness\footnote{Recall that strong stability is actually stronger than just verifying 
\eqref{resolvent}. It is known that in general, \eqref{resolvent} does not imply \eqref{powerbound} in infinite dimension, see \cite{SW}.}. As far as we know, 
the most complete answers in the discrete case are \cite{wu} (for scalar 1D problems and one time step schemes), \cite{jfcag} (for multidimensional systems 
and one time step schemes) and \cite{jfcX} (for scalar multidimensional problems and multistep schemes). In the continuous setting, that is for hyperbolic 
partial differential equations, the reader is referred to \cite{rauch,audiard,metivier2} and to references therein. All the above mentionned works are based 
on the fact that strong stability (or equivalently, the fulfillment of the so-called Uniform Kreiss-Lopatinskii Condition) provides with a sharp \emph{trace 
estimate} of the solution in terms of the data. Summarizing the methodology in the strongly stable case, the goal is to control the time derivative (the time 
difference in the discrete case) of the solution in terms of its trace. All these techniques thus break down if the considered problem is not strongly stable 
and a trace estimate is not available.

However, it has been noted that several numerical boundary conditions do not yield strongly stable problems, see for instance \cite{trefethen3}. 
As observed in \cite{trefethen3} and later made more formal in \cite{kreiss-wu}, even though the Uniform Kreiss-Lopatinskii Condition may not 
be fulfilled, it does seem that some numerical schemes remain stable in the sense that their associated (discrete) semigroup is bounded (property 
\eqref{powerbound}). This is precisely such a result that we aim at proving here, in the case where the Uniform Kreiss-Lopatinskii Condition breaks 
down because of simple, isolated eigenvalues on the unit circle\footnote{This is not the only possible breakdown for the Uniform Kreiss-Lopatinskii 
Condition, see \cite{trefethen3} or \cite{benzoni-serre} for the analogous continuous problem. However, the case we deal here with is the simplest 
and therefore the first to tackle in view of future generalizations.}. Up to our knowledge, this is the first general result of this type. Our analysis is based 
on pointwise semigroup bounds in the spirit of a long series of works initiated in \cite{ZH98} and devoted to the stability analysis of viscous shock 
profiles. We thus restrict, in this work, to finite difference approximations of the transport operator that are stable in $\ell^1(\mathbb{Z})$ (or equivalently 
$\ell^\infty(\mathbb{Z})$) without any boundary condition. By the result in \cite{Thomee}, see more recent developments in \cite{Despres,Diaconis-SaloffCoste}, 
we thus base our analysis on the dissipation Assumption \ref{hyp:1} below. This does seem restrictive at first glance, but it is very likely that our 
methodology is flexible enough to handle more general situations, up to refining some steps in the analysis. We shall explore such extensions in 
the future.

\subsection{The framework}

We consider the scalar transport equation
\begin{equation}
\label{transport}
\partial_t u \, + \, a \, \partial_x u \, = \, 0 \, ,
\end{equation}
in the half-line $\{ x \, > \, 0 \}$, and restrict from now on to the case of an incoming velocity, that is, $a>0$. The transport equation \eqref{transport} is 
supplemented with Dirichlet boundary conditions:
\begin{equation}
\label{dirichlet}
u(t,0) \, = \, 0 \, ,
\end{equation}
and a Cauchy datum at $t=0$. Our goal in this article is to explore the stability of finite difference approximations of the continuous problem \eqref{transport}, 
\eqref{dirichlet}. We thus introduce a time step $\Delta t>0$ and a space step $\Delta x>0$, assuming from now on that the ratio $\lambda \, := \, \Delta t / 
\Delta x$ is always kept fixed. The solution to \eqref{transport}, \eqref{dirichlet} is meant to be approximated by a sequence\footnote{As usual, we identify 
the sequence $(u_j^n)$ with its associated step function on the time-space grid.} $(u_j^n)$. We consider some fixed integers $r,p$ with $\min(r,p) \ge 1$. 
The \emph{interior} cells are then the intervals $[(j-1) \, \Delta x,j \, \Delta x)$ with $j \in \N^*$, and the \emph{boundary} cells are the intervals $[(\nu-1) \, 
\Delta x,\nu \, \Delta x)$ with $\nu=1-r,\dots,0$. The numerical scheme in the interior domain $\N^*$ reads:
\begin{equation}
\label{schema-int}
u_j^{n+1} \, = \, \sum_{\ell=-r}^p \, a_\ell \, u_{j+\ell}^n \, ,\quad j \ge 1 \, ,
\end{equation}
where the coefficients $a_{-r},\dots,a_p$ are real and may depend only on $\lambda$ and $a$, but not on $\Delta t$ (or $\Delta x$). The numerical boundary 
conditions that we consider in this article take the form:
\begin{equation}
\label{schema-bc}
\forall \, \nu \, = \, 1-r,\dots,0 \, ,\quad u_\nu^{n+1} \, = \, \sum_{\ell=1}^{p_b} \, b_{\ell,\nu} \, u_\ell^{n+1} \, ,
\end{equation}
where the coefficients $b_{\ell,\nu}$ in \eqref{schema-bc} are real and may also depend on $\lambda$ and $a$, but not on $\Delta t$ (or $\Delta x$). We 
assume for simplicity that the (fixed) integer $p_b$ in \eqref{schema-bc} satisfies $p_b \le p$. This is used below to simplify some minor technical details 
(when we rewrite high order \emph{scalar} recurrences as first order \emph{vectorial} recurrences).

An appropriate vector space for the stability analysis of \eqref{schema-int}-\eqref{schema-bc} is the Hilbert space $\mathcal{H}$ defined by:
\begin{equation}
\label{defH}
\mathcal{H} \, := \, \left\{  (w_j)_{j \ge 1-r} \in \ell^2 \quad / \quad \forall \, \nu \, = \, 1-r,\dots,0 \, ,\quad w_\nu \, = \, \sum_{\ell=1}^{p_b} \, b_{\ell,\nu} \, w_\ell 
\right\} \, .
\end{equation}
Sequences in $\mathcal{H}$ are assumed to be complex valued (even though, in practice, the numerical scheme \eqref{schema-int}-\eqref{schema-bc} 
applies to real sequences). Since any element $w$ of $\mathcal{H}$ is uniquely determined by its interior values (those $w_j$'s with $j \ge 1$), we use 
the following norm on $\mathcal{H}$:
$$
\forall \, w \in \mathcal{H} \, ,\quad \| w \|_{\mathcal{H}}^2 \, := \, \sum_{j \ge 1} \, |w_j|^2 \, .
$$
The numerical scheme \eqref{schema-int}-\eqref{schema-bc} can be then rewritten as:
$$
\forall \, n \in \N \, ,\quad u^{n+1} \, = \, \mathcal{T} \, u^n \, ,\quad u^0 \in \mathcal{H} \, ,
$$
where $\mathcal{T}$ is the bounded operator on $\mathcal{H}$ defined by:
\begin{equation}
\label{defT}
\forall \, w \in \mathcal{H} \, ,\quad \forall \, j \ge 1 \, ,\quad (\mathcal{T} \, w)_j \, := \, \sum_{\ell=-r}^p \, a_\ell \, w_{j+\ell} \, .
\end{equation}
Recall that a sequence in $\mathcal{H}$ is uniquely determined by its interior values so \eqref{defT} determines $\mathcal{T} \, w \in \mathcal{H}$ 
unambiguously. We introduce the following terminology.

\begin{definition}[Stability \cite{kreiss1,osher2}]
\label{def:stability}
The numerical scheme \eqref{schema-int}-\eqref{schema-bc} is said to be stable if there exists a constant $C>0$ such that, for any $f \in \mathcal{H}$, 
the solution $(u^n)_{n \in N}$ to \eqref{schema-int}-\eqref{schema-bc} with initial condition $u^0 =f$ satisfies:
$$
\sup_{n \in \N} \, \| \, u^n \, \|_{\mathcal{H}} \, \le \, C \, \| \, f \, \|_{\mathcal{H}} \, .
$$
This means equivalently that the operator $\mathcal{T}$ in \eqref{defT} is power bounded by the same constant $C$:
$$
\sup_{n \in \N} \, \| \, \mathcal{T}^{\, n} \, \|_{\mathcal{H} \rightarrow \mathcal{H}} \, \le \, C \, .
$$
\end{definition}

\noindent Our goal in this article is to show that the scheme \eqref{schema-int}-\eqref{schema-bc} is stable under some spectral assumptions on 
the operator $\mathcal{T}$.

\subsection{Assumptions and main result}

We make two major assumptions: one on the finite difference scheme \eqref{schema-int}, and one on the compatibility between the scheme 
\eqref{schema-int} and the numerical boundary conditions \eqref{schema-bc}.

\begin{assumption}
\label{hyp:1}
The finite difference approximation \eqref{schema-int} is consistent with the transport equation \eqref{transport}:
\begin{equation}
\label{hyp:consistance}
\sum_{\ell=-r}^p \, a_\ell \, = \, 1 \, ,\quad \sum_{\ell=-r}^p \, \ell \, a_\ell \, = \, -\lambda \, a \, < \, 0 \, ,\quad \text{\rm (consistency).}
\end{equation}
Moreover, the coefficients $a_\ell$ in \eqref{schema-int} satisfy $a_{-r} \, a_p \neq 0$ and the \emph{dissipativity} condition:
\begin{equation}
\label{hyp:stabilite1}
\forall \, \theta \in [ - \, \pi \, , \, \pi ] \setminus \{ 0 \} \, ,\quad \left| \, \sum_{\ell=-r}^p \, a_\ell \, {\rm e}^{\, \mathbf{i} \, \ell \, \theta} \, \right| \, < \, 1 \, ,
\end{equation}
and for some nonzero integer $\mu$ and some positive real number $\beta>0$, there holds:
\begin{equation}
\label{hyp:stabilite2}
\sum_{\ell=-r}^p \, a_\ell \, {\rm e}^{\, \mathbf{i} \, \ell \, \theta} \, = \, \exp \left( - \, \mathbf{i} \, \lambda \, a \, \theta \, - \, \beta \, \theta^{\, 2 \, \mu} \, 
+ \, O \Big( \theta^{\, 2 \, \mu+1} \Big) \right) \, ,
\end{equation}
as $\theta$ tends to $0$.
\end{assumption}

\noindent An important consequence of Assumption \ref{hyp:1} is the following Bernstein type inequality, which we prove in Appendix \ref{sec:appendix}.

\begin{lemma}
\label{lem:Bernstein}
Under Assumption \ref{hyp:1}, there holds $\lambda \, a \, < \, r$.
\end{lemma}

The relevance of \eqref{hyp:stabilite2} for the $\ell^1$ stability of \eqref{schema-int} on $\Z$ is the major result in \cite{Thomee} (see 
\cite{CF1,Despres,Diaconis-SaloffCoste} for recent developments in this direction). This stability property will greatly simplify the final steps of the 
proof of our main result, which is Theorem \ref{thm1} below. Relaxing \eqref{hyp:stabilite1} and \eqref{hyp:stabilite2} in order to encompass a wider 
class of finite difference schemes is postponed to some future works. We now state two Lemma whose proofs, which are relatively standard, can 
also be found in Appendix~\ref{sec:appendix}. These two Lemma will allow us to introduce our second spectral assumption on the operator $\T$.

\begin{lemma}
\label{lem:1}
There exists a constant $c_0>0$ such that, if we define the set:
$$
\mathcal{C} \, := \, \Big\{ 
\rho \, {\rm e}^{\, \mathbf{i} \, \varphi} \in \C \, / \, \varphi \in [ - \, \pi \, , \, \pi ] \quad \text{\rm and} \quad 
0 \, \le \, \rho \, \le \, 1 \, - \, c_0 \, \varphi^{\, 2 \, \mu} \Big\} \, ,
$$
then $\mathcal{C}$ is a compact star-shaped subset of $\Dbar$, and the curve:
\begin{equation}
\label{curve-spectrum}
\left\{ \sum_{\ell=-r}^p \, a_\ell \, {\rm e}^{\, \mathbf{i} \, \ell \, \theta} \, / \, \theta \in [ - \, \pi \, , \, \pi ] \right\}
\end{equation}
is contained in $\mathcal{C}$.
\end{lemma}

The above Lemma~\ref{lem:1} provides an estimate on the location of the essential spectrum of the operator $\T$ and shows that it is contained in 
$\mathcal{C}$ (see the reminder below on the spectrum of Toeplitz operators). Next, we introduce the following matrix:
\begin{equation}
\label{defM}
\forall \, z \in \C \, ,\quad \M (z) \, := \, \begin{pmatrix}
\dfrac{\delta_{p-1,0} \, z \, - \, a_{p-1}}{a_p} & \dots & \dots & \dfrac{\delta_{-r,0} \, z \, - \, a_{-r}}{a_p} \\
1 & 0 & \dots & 0 \\
0 & \ddots & \ddots & \vdots \\
0 & 0 & 1 & 0 \end{pmatrix} \in \mathcal{M}_{p+r}(\C) \, .
\end{equation}
Since $\min (r,p) \ge 1$, the upper right coefficient of $\M(z)$ is always nonzero (it equals $-a_{-r}/a_p$), and $\M(z)$ is invertible. We shall repeatedly 
use the inverse matrix $\M(z)^{-1}$ in what follows.

\begin{lemma}[Spectral splitting]
\label{lem:2}
Let $z \in \C$ and let the matrix $\M(z)$ be defined as in \eqref{defM}. Let the set $\mathcal{C}$ be defined by Lemma \ref{lem:1}. Then 
for $z \not \in \mathcal{C}$, $\M(z)$ has:
\begin{itemize}
 \item no eigenvalue on $\cercle$,
 \item $r$ eigenvalues in $\D \setminus \{ 0 \}$,
 \item $p$ eigenvalues in $\U$ (eigenvalues are counted with multiplicity).
\end{itemize}
Furthermore, $\M(1)$ has $1$ as a simple eigenvalue, it has $r-1$ eigenvalues in $\D$ and $p$ eigenvalues in $\U$.
\end{lemma}

We introduce some notation. For $z \not \in \mathcal{C}$, Lemma \ref{lem:2} shows that the so-called stable subspace, which is spanned by the 
generalized eigenvectors of $\M(z)$ associated with eigenvalues in $\D$, has constant dimension $r$. We let $\E^s(z)$ denote the stable subspace 
of $\M(z)$ for $z \not \in \mathcal{C}$. Because of the spectral splitting shown in Lemma \ref{lem:2}, $\E^s(z)$ depends holomorphically on $z$ in 
the complementary set of $\mathcal{C}$. We can therefore find, near every point $\underline{z} \not \in \mathcal{C}$, a basis $e_1(z),\dots,e_r(z)$ 
of $\E^s(z)$ that depends holomorphically on $z$. Similarly, the unstable subspace, which is spanned by the generalized eigenvectors of $\M(z)$ 
associated with eigenvalues in $\U$, has constant dimension $p$. We denote it by $\E^u(z)$, and it also depends holomorphically on $z$ in the 
complementary set of $\mathcal{C}$. With obvious notation, the projectors associated with the decomposition:
$$
\forall \, z \not \in \mathcal{C} \, ,\quad \C^{p+r} \, = \, \E^s(z) \oplus \E^u(z) \, ,
$$
are denoted $\pi^s(z)$ and $\pi^u(z)$.

Let us now examine the situation close to $z=1$. Since $1$ is a simple eigenvalue of $\M(1)$, we can extend it holomorphically to a simple eigenvalue 
$\kappa(z)$ of $\M(z)$ in a neighborhood of $1$. This eigenvalue is associated with the eigenvector:
$$
E(z) \, := \, \begin{bmatrix}
\, \kappa(z)^{p+r-1} \, \\
\vdots \\
\kappa(z) \\
1 \end{bmatrix} \in \C^{p+r} \, ,
$$
which also depends holomorphically on $z$ in a neighborhood of $1$. Furthermore, the unstable subspace $\E^u(1)$ associated with eigenvalues in 
$\U$ has dimension $p$. It can be extended holomorphically to a neighborhood of $1$ thanks to the Dunford formula for spectral projectors. This 
holomorphic extension coincides with the above definition for $\E^u(z)$ if $z$ is close to $1$ and $z \not \in \mathcal{C}$. Eventually, the stable 
subspace of $\M(1)$ associated with eigenvalues in $\D$ has dimension $r-1$. For the sake of clarity, we denote it by $\E^{ss}(1)$ (the double $s$ 
standing for \emph{strongly stable}). Using again the Dunford formula for spectral projectors, we can extend this ``strongly stable'' subspace 
holomorphically with respect to $z$; for $z$ close to $1$, $\E^{ss}(z)$ has dimension $r-1$ and is either all or a hyperplane within the stable 
subspace of $\M(z)$. Namely, the situation has no ambiguity: for $z \not \in \mathcal{C}$ close to $1$, the eigenvalue $\kappa(z)$ necessarily 
belongs to $\D$ and the stable subspace $\E^s(z)$ of $\M(z)$ (which has been defined above and has dimension $r$) splits as:
\begin{equation}
\label{decomposition-Es}
\E^s(z) \, = \, \E^{ss}(z) \, \oplus \, \text{\rm Span } E(z) \, .
\end{equation}
Since the right hand side in \eqref{decomposition-Es} depends holomorphically on $z$ in a whole neighborhood of $1$ and not only in $\mathcal{C}^c$, 
the stable subspace $\E^s(z)$ extends holomorphically to a whole neighborhood of $1$ as an invariant subspace of dimension $r$ for $\M(z)$. In 
particular, we shall feel free to use below the notation $\E^s(1)$ for the $r$-dimensional vector space:
\begin{equation}
\label{decomposition-Es1}
\E^s(1) \, := \, \E^{ss}(1) \, \oplus \, \text{\rm Span } \begin{bmatrix}
1 \\
1 \\
\vdots \\
\, 1 \, \end{bmatrix} \, ,
\end{equation}
which is, in our case, the direct sum of the stable and central subspaces of $\M(1)$.

For future use, it is convenient to introduce the following matrix:
\begin{equation}
\label{defB}
\mathcal{B} \, := \, \begin{pmatrix}
0 & \cdots & 0 & - \, b_{p_b,0} & \cdots & - \, b_{1,0} & 1 & 0 & \cdots & 0 \\
\vdots &  & \vdots & \vdots &  & \vdots & 0  & \vdots & \ddots & \vdots \\
\vdots &  & \vdots & \vdots &  & \vdots & \vdots  & \ddots & \ddots & 0 \\
0 & \cdots & 0 & - \, b_{p_b,1-r} & \cdots & - \, b_{1,1-r} & 0 & \cdots & 0 & 1 \end{pmatrix} \in \mathcal{M}_{r,p+r}(\R) \, .
\end{equation}
We can now state our final assumption.

\begin{assumption}
\label{hyp:2}
For any $z \in \U \cup \{ 1 \}$, there holds:
$$
\C^{p+r} \, = \, \text{\rm Ker } \mathcal{B} \, \oplus \, \E^s (z) \, ,
$$
or, in other words, $\mathcal{B}|_{\E^s (z)}$ is an isomorphism from $\E^s (z)$ to $\C^r$. Moreover, choosing a holomorphic basis 
$e_1(z),\dots,e_r(z)$ of $\E^s (z)$ near every point $\underline{z} \in \cercle \setminus \{ 1 \}$, the function:
$$
\Delta \quad : \quad z \, \longmapsto \, \det \begin{bmatrix}
\, \mathcal{B} \, e_1(z) \, & \, \cdots \, & \, \mathcal{B} \, e_r(z) \, \end{bmatrix}
$$
has finitely many simple zeroes in $\cercle \setminus \{ 1 \}$.
\end{assumption}

Let us recall that for $z \in \Ubar \setminus \{ 1 \}$, $\E^s (z)$ denotes the stable subspace of the matrix $\M(z)$ in \eqref{defM} since then 
$z \in \mathcal{C}^c$. At the point $z=1$, $\E^s (1)$ denotes the holomorphic extension of $\E^s (z)$ at $1$ and it is furthermore given by 
\eqref{decomposition-Es1}.

Of course, the function $\Delta$ in Assumption \ref{hyp:2} depends on the choice of the (holomorphic) basis $e_1(z)$, \dots, $e_r(z)$ of 
$\E^s (z)$. However, the location of its zeroes and their multiplicity does not depend on that choice, which means that Assumption \ref{hyp:2} 
is an intrinsic property of the operator $\mathcal{T}$. We shall refer later on to the function $\Delta$ as the \emph{Lopatinskii determinant} 
associated with \eqref{schema-int}-\eqref{schema-bc}. It plays the role of a characteristic polynomial for $\mathcal{T}$ which detects the 
eigenvalues in $\mathcal{C}^c$. This object already appears in \cite{kreiss1,osher1,osher2}. Its analogue in the study of discrete shock 
profiles is the so-called Evans function, see \cite{godillon}. Our main result is the following.

\begin{theorem}
\label{thm1}
Under Assumptions \ref{hyp:1} and \ref{hyp:2}, the operator $\mathcal{T}$ in \eqref{defT} is power bounded, that is, the numerical scheme 
\eqref{schema-int}-\eqref{schema-bc} is stable.
\end{theorem}

\noindent If the function $\Delta$ in Assumption \ref{hyp:2} does not vanish on $\Ubar$, the Uniform Kreiss-Lopatinskii Condition is said to hold 
and the main result in \cite{wu} implies that $\mathcal{T}$ is power bounded, see also \cite{kreiss1,osher1,osher2}. The novelty here is to allow 
$\Delta$ to vanish on $\cercle$. The Uniform Kreiss-Lopatinskii Condition thus breaks down. Power boundedness of $\mathcal{T}$ in this case 
was conjectured in \cite{trefethen3,kreiss-wu}.

The remainder of this article is organized as follows. The proof of Theorem \ref{thm1} follows the same strategy as in \cite{CF1}. In Section 
\ref{sect:spectral}, we clarify the location of the spectrum of $\mathcal{T}$ and give accurate bounds on the so-called \emph{spatial} Green's 
function (that is, the Green's function for the operator $z\, I-\mathcal{T}$ with $z \not \in \sigma (\mathcal{T})$). This preliminary analysis is 
used in Section \ref{sect:proof} to give an accurate description of the so-called \emph{temporal} Green's function (that is, the Green's function 
for the original problem \eqref{schema-int}-\eqref{schema-bc}). Power boundedness of $\mathcal{T}$ easily follows by classical inequalities. 
An example of operator for which Theorem \ref{thm1} applies is given in Section \ref{sect:example}.

\section{Spectral analysis}
\label{sect:spectral}

For later use, we let $\underline{z}_1,\dots,\underline{z}_K \in \cercle \setminus \{ 1 \}$ denote the pairwise distinct roots of the Lopatinskii determinant 
$\Delta$ introduced in Assumption \ref{hyp:2}. We recall that these roots are simple. We first locate the spectrum of the operator $\mathcal{T}$ and 
then give an accurate description of the so-called spatial Green's function. Precise definitions are provided below.

\subsection{A reminder on the spectrum of Toeplitz operators}

The operator $\mathcal{T}$ is a finite rank (hence compact) perturbation of the Toeplitz operator on $\ell^2(\N)$ represented by the semi-infinite matrix:
$$
\begin{pmatrix}
a_0 & \cdots & a_p & 0 & \cdots & \cdots & \cdots \\
\vdots & \ddots & & \ddots & \ddots & & \\
a_{-r} & \cdots & a_0 & \cdots & a_p & \ddots & \\
0 & \ddots & & \ddots & & \ddots & \ddots \\
\vdots & \ddots & \ddots & & \ddots & & \ddots \\
\end{pmatrix} \, .
$$
Therefore $\mathcal{T}$ shares the same essential spectrum as the Toeplitz operator \cite{Conway}. (The latter Toeplitz operator corresponds to enforcing 
the Dirichlet boundary conditions $u_{1-r}^{n+1}=\cdots=u_0^{n+1}=0$ instead of the more general form \eqref{schema-bc}). The spectrum of Toeplitz 
operators is well-known, see for instance \cite{Duren} and further developments in \cite{TE}. The resolvent set of the above Toeplitz operator consists of 
all points $z \in \C$ that do not belong to the curve \eqref{curve-spectrum} and that have index $0$ with respect to it. Moreover, any point on the curve 
\eqref{curve-spectrum} is in the essential spectrum. In the particular case we are interested in, Assumption \ref{hyp:1} implies that the essential spectrum 
of $\mathcal{T}$ is located in the set $\mathcal{C}$ defined by Lemma \ref{lem:1} and that $1$ belongs to the essential spectrum of $\mathcal{T}$. There 
remains to clarify the point spectrum of $\mathcal{T}$. The situation which we consider here and that is encoded in Assumption \ref{hyp:2} is that where 
the finite rank perturbation of the Toeplitz operator generates finitely many simple eigenvalues on the unit circle (there may also be eigenvalues within 
$\mathcal{C}$ but we are mainly concerned here with the eigenvalues of largest modulus). A precise statement is the following.

\begin{lemma}[The resolvant set]
\label{lem:3}
Let the set $\mathcal{C}$ be defined by Lemma \ref{lem:1}. Then there exists $\varepsilon>0$ such that $(\mathcal{C}^c \setminus \{ \underline{z}_1, \dots, 
\underline{z}_K \}) \cap \{ \zeta \in \C \, / \, |\zeta| \, > \, 1-\varepsilon \}$ is contained in the resolvant set of the operator $\mathcal{T}$. Moreover, each 
zero $\underline{z}_k$ of the Lopatinskii determinant is an eigenvalue of $\mathcal{T}$.
\end{lemma}

\begin{proof}
The proof of Lemma \ref{lem:3} is first useful to clarify the location of the spectrum of the operator $\mathcal{T}$ and it is also useful to introduce 
some of the tools used in the construction of the spatial Green's function which we shall perform below.

Let therefore $z \in \mathcal{C}^c \setminus \{ \underline{z}_1,\dots,\underline{z}_K \}$ and let $f \in \mathcal{H}$. We are going to explain why we 
can uniquely solve the equation:
\begin{equation}
\label{resolvant}
(z \, I \, - \, \mathcal{T}) \, w \, = \, f \, ,
\end{equation}
with $w \in \mathcal{H}$ (up to assuming $|z|>1-\varepsilon$ for some sufficiently small $\varepsilon>0$). Using the definitions \eqref{defH} and \eqref{defT}, 
we wish to solve the system:
$$
\begin{cases}
z \, w_j \, - \, \sum_{\ell=-r}^p \, a_\ell \, w_{j+\ell} \, = \, f_j \, , & j \ge 1 \, ,\\
w_\nu \, = \, \sum_{\ell=1}^{p_b} \, b_{\ell,\nu} \, w_\ell \, ,& \nu \, = \, 1-r,\dots,0 \, .
\end{cases}
$$
We introduce, for any $j \ge 1$, the augmented vector:
$$
W_j \, := \, \begin{bmatrix}
\, w_{j+p-1} \, \\
\vdots \\
\, w_{j-r} \, \end{bmatrix} \in \C^{p+r} \, ,
$$
which must satisfy the problem\footnote{This is the place where we use the assumption $p_b \le p$ in order to rewrite the numerical boundary conditions 
as a linear constraint on the first element $W_1$ of the sequence $(W_j)_{j \ge 1}$. The case $p_b>p$ can be dealt with quite similarly but is just heavier 
in terms of notation.}:
\begin{equation}
\label{dyn-spatiale}
\begin{cases}
W_{j+1} \, = \, \M(z) \, W_j -a_p^{-1} \, f_j \, {\bf e} \, ,& j \ge 1 \, ,\\
\mathcal{B} \, W_1 \, = \, 0 \, ,& 
\end{cases}
\end{equation}
where we have used the notation ${\bf e}$ to denote the first vector of the canonical basis of $\C^{p+r}$, namely:
$$
{\bf e} \, := \, \begin{bmatrix}
\, 1 \, \\
\, 0 \, \\
\vdots \\
\, 0 \, \end{bmatrix} \in \C^{p+r} \, .
$$

Our goal now is to solve the spatial dynamical problem \eqref{dyn-spatiale}. Since $z \in \mathcal{C}^c$, we know that $\M(z)$ enjoys a hyperbolic dichotomy 
between its unstable and stable eigenvalues. We first solve for the unstable components of the sequence $(W_j)_{j \ge 1}$ by integrating from $+\infty$ to any 
integer $j \ge 1$, which gives:
\begin{equation}
\label{composanteinstablej}
\forall \, j \ge 1 \, ,\quad \pi^u(z) \, W_j \, := \, a_p^{-1} \, \sum_{\ell \ge 0} \, f_{j+\ell} \, \M(z)^{-1-\ell} \, \pi^u(z) \, {\bf e} \, .
\end{equation}
In particular, we get the ``initial value'':
\begin{equation}
\label{composanteinstable1}
\pi^u(z) \, W_1 \, := \, a_p^{-1} \, \sum_{\ell \ge 0} \, f_{1+\ell} \, \M(z)^{-1-\ell} \, \pi^u(z) \, {\bf e} \, .
\end{equation}

The initial value for the stable components is obtained by using Assumption \ref{hyp:2}. Namely, if $z \in \U$, we know that the linear operator 
$\mathcal{B}|_{\E^s(z)}$ is an isomorphism, and this property remains true near every point of $\cercle \setminus \{ 1 \}$ except at the $\underline{z}_k$'s. 
We can thus find some $\varepsilon>0$ such that, for any $z \in \mathcal{C}^c \setminus \{ \underline{z}_1,\dots,\underline{z}_K \}$ verifying 
$|z|>1-\varepsilon$, the linear operator $\mathcal{B}|_{\E^s(z)}$ is an isomorphism. For such $z$'s, we can therefore define the vector $\pi^s(z) 
\, W_1 \in \E^s(z)$ through the formula:
\begin{equation}
\label{composantestable1}
\pi^s(z) \, W_1 \, := -\, a_p^{-1} \, \Big( \mathcal{B}|_{\E^s(z)} \Big)^{-1} \, \mathcal{B} \, \sum_{\ell \ge 0} \, f_{1+\ell} \, \M(z)^{-1-\ell} \, \pi^u(z) \, {\bf e} \, ,
\end{equation}
which is the only way to obtain both the linear constraint $\mathcal{B} \, W_1=0$ and the decomposition $W_1=\pi^s(z) \, W_1+\pi^u(z) \, W_1$ in 
agreement with \eqref{composanteinstable1}. Once we have determined the stable components $\pi^s(z) \, W_1$ of the initial value $W_1$, the only 
possible way to solve \eqref{dyn-spatiale} for the stable components is to set:
\begin{equation}
\label{composantestablej}
\forall \, j \ge 1 \, ,\quad \pi^s(z) \, W_j \, := \, \M(z)^{j-1} \, \pi^s(z) \, W_1 \, - \, a_p^{-1} \, \sum_{\ell = 1}^{j-1} \, f_{\ell} \, \M(z)^{j-1-\ell} \, \pi^s(z) \, {\bf e} \, .
\end{equation}

Since the sequences $(\M(z)^{-\ell} \, \pi^u(z))_{\ell \ge 1}$ and $(\M(z)^{\ell} \, \pi^s(z))_{\ell \ge 1}$ are exponentially decreasing, we can define a 
solution $(W_j)_{j \ge 1} \in \ell^2$ to \eqref{dyn-spatiale} by decomposing along the stable and unstable components and using the defining equations 
\eqref{composanteinstablej} and \eqref{composantestablej}. This provides us with a solution $w \in \mathcal{H}$ to the equation \eqref{resolvant} by 
going back to the scalar components of each vector $W_j$. Such a solution is necessarily unique since if $w \in \mathcal{H}$ is a solution to 
\eqref{resolvant} with $f=0$, then the augmented vectorial sequence $(W_j)_{j \ge 1} \in \ell^2$ satisfies:
$$
\begin{cases}
W_{j+1} \, = \, \M(z) \, W_j \, ,& j \ge 1 \, ,\\
\mathcal{B} \, W_1 \, = \, 0 \, .& 
\end{cases}
$$
This means that the vector $W_1$ belongs to $\E^s(z)$ and to the kernel of the matrix $\mathcal{B}$, and therefore vanishes. Hence the whole sequence 
$(W_j)_{j \ge 1} \in \ell^2$ vanishes. We have thus shown that $z$ belongs to the resolvant set of $\mathcal{T}$.

The fact that each $\underline{z}_k$ is an eigenvalue of $\mathcal{T}$ follows from similar arguments. At a point $\underline{z}_k$, the intersection 
$\text{\rm Ker } \mathcal{B} \, \cap \, \E^s (\underline{z}_k)$ is not trivial, so we can find a nonzero vector $\underline{W}_1 \in \text{\rm Ker } \mathcal{B}$ 
for which the sequence $(\underline{W}_j)_{j \ge 1}$ defined by:
$$
\forall \, j \ge 1 \, ,\quad \underline{W}_{j+1} \, := \, \M(\underline{z}_k) \, \underline{W}_j \, ,
$$
is square integrable (it is even exponentially decreasing). Going back to scalar components, this provides with a nonzero solution to the eigenvalue problem:
$$
\begin{cases}
\underline{z}_k \, w_j \, - \, \sum_{\ell=-r}^p \, a_\ell \, w_{j+\ell} \, = \, 0 \, , & j \ge 1 \, ,\\
w_\nu \, = \, \sum_{\ell=1}^{p_b} \, b_{\ell,\nu} \, w_\ell \, ,& \nu \, = \, 1-r,\dots,0 \, .
\end{cases}
$$
The proof of Lemma \ref{lem:3} is complete.
\end{proof}

We are now going to define and analyze the so-called spatial Green's function. The main point, as in \cite{ZH98,godillon,CF1} and related works, is to be 
able to ``pass through'' the essential spectrum close to $1$ and extend the spatial Green's function holomorphically to a whole neighborhood of $1$. This 
was already achieved with accurate bounds in \cite{CF1} on the whole line $\Z$ (with no numerical boundary condition) and we apply similar arguments 
here, while adding the difficulty of the eigenvalues on $\cercle$. Near all such eigenvalues, we isolate the precise form of the singularity in the Green's 
function and show that the remainder admits a holomorphic extension at the eigenvalue. All these arguments are made precise in the following paragraph.

\subsection{The spatial Green's function}

For any $j_0 \ge 1$, we let $\boldsymbol{\delta}_{j_0}$ denote the only element of the space $\mathcal{H}$ in \eqref{defH} that satisfies:
$$
\forall \, j \, \ge \, 1 \, ,\quad \big( \boldsymbol{\delta}_{j_0} \big)_j \, := \, \begin{cases}
1 \, ,& \text{\rm if $j \, = \, j_0$,} \\
0 \, ,& \text{\rm otherwise.}
\end{cases}
$$
The boundary values of $\boldsymbol{\delta}_{j_0}$ are defined accordingly. Then as long as $z$ belongs to the resolvant set of the operator $\mathcal{T}$, 
the spatial Green's function, which we denote $G_z (\cdot,\cdot)$ is defined by the relation:
\begin{equation}
\label{defGzj}
\forall \, j_0 \, \ge \, 1 \, ,\quad \big( z \, I \, - \, \mathcal{T} \big) \, G_z (\cdot,j_0) \, = \, \boldsymbol{\delta}_{j_0} \, ,
\end{equation}
together with the numerical boundary conditions $G_z (\cdot,j_0) \in \mathcal{H}$. We give below an accurate description of $G_z$ in order to later obtain 
an accurate description of the temporal Green's function, that is obtained by applying the iteration \eqref{schema-int}-\eqref{schema-bc} to the initial condition 
$\boldsymbol{\delta}_{j_0}$. The analysis of the spatial Green's function splits between three cases:
\begin{itemize}
 \item The behavior near regular points (away from the spectrum of $\mathcal{T}$),
 \item The behavior near the point $1$ (the only point where the essential spectrum of $\mathcal{T}$ meets $\cercle$),
 \item The behavior near the eigenvalues $\underline{z}_1,\dots,\underline{z}_K$.
\end{itemize}
Let us start with the easiest case.

\begin{lemma}[Bounds away from the spectrum]
\label{lem:4}
Let $\underline{z} \in \Ubar \setminus \{ 1,\underline{z}_1,\dots,\underline{z}_K \}$. Then there exists an open ball $B_r(\underline{z})$ centered at 
$\underline{z}$ and there exist two constants $C>0$, $c>0$ such that, for any couple of integers $j,j_0 \ge 1$, there holds:
$$
\forall \, z \in B_r(\underline{z}) \, ,\quad \big| \, G_z(j,j_0) \, \big| \, \le \, C \, \exp \, \big( - \, c \, |j \, - \, j_0| \, \big) \, .
$$
\end{lemma}

\begin{proof}
Almost all ingredients have already been set in the proof of Lemma \ref{lem:3}. Let therefore $\underline{z} \in \Ubar \setminus \{ 1,\underline{z}_1, \dots, 
\underline{z}_K \}$, and let us first fix $r>0$ small enough such that the closed ball $\overline{B_r(\underline{z})}$ is contained both in $\mathcal{C}^c$ 
and in the resolvant set of $\mathcal{T}$. All complex numbers $z$ below are assumed to lie within $\overline{B_r(\underline{z})}$. Then the problem 
\eqref{defGzj} can be recast under the vectorial form \eqref{dyn-spatiale} with:
$$
\forall \, j \ge 1 \, ,\quad f_j \, := \, \begin{cases}
1 \, ,& \text{\rm if $j \, = \, j_0$,} \\
0 \, ,& \text{\rm otherwise.}
\end{cases}
$$
Let us therefore consider the spatial dynamics problem \eqref{dyn-spatiale} with the above Dirac mass type source term. The unstable components of the 
sequence $(W_j)_{j \ge 1}$ solution to \eqref{dyn-spatiale} are given by \eqref{composanteinstablej}, which gives here:
$$
\forall \, j \ge 1 \, ,\quad \pi^u (z) \, W_j \, = \, \begin{cases}
0  \, ,& \text{\rm if $j \, > \, j_0$,} \\
a_p^{-1} \, \M(z)^{-(j_0+1-j)} \, \pi^u(z) \, {\bf e} \, ,& \text{\rm if $1 \, \le \, j \, \le \, j_0$.}
\end{cases}
$$
In particular, we get the following uniform bounds with respect to $z,j,j_0$:
\begin{equation}
\label{borneslem4-1}
\forall \, z \in \overline{B_r(\underline{z})} \, ,\quad \forall \, j \ge 1 \, ,\quad \Big| \pi^u (z) \, W_j \Big| \, \le \, \begin{cases}
0  \, ,& \text{\rm if $j \, > \, j_0$,} \\
C \, \exp (- \, c \, (j_0-j)) \, ,& \text{\rm if $1 \, \le \, j \, \le \, j_0$.}
\end{cases}
\end{equation}

The initial value $\pi^s (z) \, W_1$ of the stable components is then obtained by the relation \eqref{composantestable1}, which immediately gives the 
bound\footnote{Here we use the fact that the linear map $\mathcal{B}|_{\E^s(z)}$ is an isomorphism for all $z \in \overline{B_r(\underline{z})}$.}:
$$
\Big| \pi^s (z) \, W_1 \Big| \, \le \, C \, \exp (- \, c \, j_0) \, .
$$
The stable components are then determined for any integer $j \ge 1$ by the general formula \eqref{composantestablej}, which gives here:
$$
\forall \, j \ge 1 \, ,\quad \pi^s (z) \, W_j \, = \, \begin{cases}
\M(z)^{j-1} \, \pi^s(z) \, W_1 \, ,& \text{\rm if $1 \, \le \, j \, \le \, j_0$,} \\
\M(z)^{j-1} \, \pi^s(z) \, W_1 -a_p^{-1} \, \M(z)^{j-j_0-1} \, \pi^s(z) \, {\bf e} \, ,& \text{\rm if $j \, > \, j_0$.}
\end{cases}
$$
By using the exponential decay of the sequence $(\M(z)^{\, j} \, \pi^s(z))_{j \ge 1}$, we get the following bounds for the stable components:
\begin{equation}
\label{borneslem4-2}
\forall \, j \ge 1 \, ,\quad \Big| \pi^s (z) \, W_j \Big| \, \le \, \begin{cases}
C \, \exp (- \, c \, (j_0+j)) \, ,& \text{\rm if $1 \, \le \, j \, \le \, j_0$,} \\
C \, \exp (- \, c \, (j_0+j)) \, + \, C \, \exp (- \, c \, (j-j_0)) \, ,& \text{\rm if $j \, > \, j_0$.}
\end{cases}
\end{equation}

Adding \eqref{borneslem4-1} with \eqref{borneslem4-2}, and examining for each situation which among the terms is the largest, we get the conclusion 
of Lemma \ref{lem:4} (recall that the scalar component $G_z(j,j_0)$ is just one among the coordinates of the vector $W_j$ considered above).
\end{proof}

We are now going to examine the behavior of the spatial Green's function $G_z$ close to $1$. Let us first recall that the exterior $\U$ of the unit 
disk belongs to the resolvant set of $\mathcal{T}$. Hence, for any $j_0 \ge 1$, the sequence $G_z(\cdot,j_0)$ is well-defined in $\mathcal{H}$ 
for $z \in \U$. Lemma \ref{lem:5} below shows that each individual sequence $G_z(j,j_0)$ can be holomorphically extended to a whole 
neighborhood of $1$.

\begin{lemma}[Bounds close to $1$]
\label{lem:5}
There exists an open ball $B_\varepsilon(1)$ centered at $1$ and there exist two constants $C_1>0$ and $c_1>0$ such that, for any couple of integers 
$(j,j_0)$, the component $G_z(j,j_0)$ defined on $B_\varepsilon(1)\cap \U$ extends holomorphically to the whole ball $B_\varepsilon(1)$ with respect to 
$z$, and the holomorphic extension satisfies the bound:
$$
\forall \, z \in B_\varepsilon(1) \, ,\quad \big| \, G_z(j,j_0) \, \big| \, \le \, \begin{cases}
C_1 \, \exp \, \big( -c_1 \, |j \, - \, j_0| \, \big) \, ,& \text{\rm if $1 \, \le \, j \, \le \, j_0$,} \\
C_1 \, \Big| \kappa(z) \Big|^{|j \, - \, j_0|} \, ,& \text{\rm if $j \, > \, j_0$,}
\end{cases}
$$
where $\kappa(z)$ denotes the (unique) holomorphic eigenvalue of $\M(z)$ that satisfies $\kappa(1)=1$.
\end{lemma}

\begin{proof}
Most ingredients of the proof are similar to what we have already done in the proof of Lemma \ref{lem:4}. The novelty is that there is one stable 
component which behaves more and more singularly as $z \in \U$ gets close to $1$ since one stable eigenvalue, namely $\kappa(z)$, gets close 
to $\cercle$ (its exponential decay is thus weaker and weaker). We thus recall that on some suitably small neighborhood $B_\varepsilon(1)$ of $1$, 
we have the (holomorphic in $z$) decomposition:
$$
\C^{p+r} \, = \, \E^u (z) \, \oplus \, \E^{ss}(z) \, \oplus \, \text{\rm Span } E(z) \, ,
$$
where all the above spaces are invariant by $\M(z)$, the spectrum of $\M(z)$ restricted to $\E^u (z)$ lies in $\U$, the spectrum of $\M(z)$ restricted 
to $\E^{ss} (z)$ lies in $\D$, and $E(z)$ is an eigenvector for $\M(z)$ associated with the eigenvalue $\kappa(z)$. With obvious notation, we use the 
corresponding decomposition:
$$
X \, = \, \pi^u(z) \, X \, + \, \pi^{ss}(z) \, X \, + \, \mu \, E(z) \, .
$$

Let us from now on consider some complex number $z \in B_\varepsilon(1) \cap \U$ so that the Green's function $G_z(\cdot,j_0)$ is well-defined in 
$\mathcal{H}$ for any $j_0 \ge 1$. As in the proof of Lemma \ref{lem:4}, the Green's function is defined by solving the spatial dynamics problem 
\eqref{dyn-spatiale} with the Dirac mass datum:
$$
\forall \, j \ge 1 \, ,\quad f_j \, := \, \begin{cases}
1 \, ,& \text{\rm if $j \, = \, j_0$,} \\
0 \, ,& \text{\rm otherwise.}
\end{cases}
$$
The unstable components are uniquely determined by:
$$
\forall \, j \ge 1 \, ,\quad \pi^u (z) \, W_j \, = \, \begin{cases}
0  \, ,& \text{\rm if $j \, > \, j_0$,} \\
a_p^{-1} \, \M(z)^{-(j_0+1-j)} \, \pi^u(z) \, {\bf e} \, ,& \text{\rm if $1 \, \le \, j \, \le \, j_0$,}
\end{cases}
$$
and we readily observe that the latter right hand side depends holomorphically on $z$ in the whole ball $B_\varepsilon(1)$ and not only in 
$B_\varepsilon(1) \cap \U$. This already allows to extend the unstable components $\pi^u (z) \, W_j$ to $B_\varepsilon(1)$, with the corresponding 
uniform bound similar to \eqref{borneslem4-1}, that is:
\begin{equation}
\label{borneslem5-1}
\forall \, z \in B_\varepsilon(1) \, ,\quad \forall \, j \ge 1 \, ,\quad \Big| \pi^u (z) \, W_j \Big| \, \le \, \begin{cases}
0  \, ,& \text{\rm if $j \, > \, j_0$,} \\
C \, \exp (- \, c \, (j_0-j)) \, ,& \text{\rm if $1 \, \le \, j \, \le \, j_0$.}
\end{cases}
\end{equation}

We can then use the fact that $\mathcal{B}|_{\E^s(1)}$ is an isomorphism from $\E^s(1)$ to $\C^r$, which implies that, up to restricting the radius 
$\varepsilon$, the matrix $\mathcal{B}$ restricted to the holomorphically extended stable subspace:
$$
\E^{ss} (z) \, \oplus \, \text{\rm Span } E(z) \, ,
$$
is an isomorphism. We can thus uniquely determine some vector $\pi^{ss}(z) \, W_1 \in \E^{ss}(z)$ and a scalar $\mu_1$ such that:
$$
\mathcal{B} \, \Big( \pi^{ss}(z) \, W_1 \, + \, \mu_1 \, E(z) \Big) \, = -\, a_p^{-1} \, \mathcal{B}  \, \M(z)^{-j_0} \, \pi^u(z) \, {\bf e} \, .
$$
In particular, we have the bound:
$$
\forall \, z \in B_\varepsilon(1) \, ,\quad  \Big| \pi^{ss} (z) \, W_1 \Big| \, + \, |\mu_1| \, \le \, C \, \exp (- \, c \, j_0) \, .
$$

For $z \in B_\varepsilon(1) \cap \U$, the strongly stable components of $(W_j)_{j \ge 1}$ are then defined by the formula:
$$
\forall \, j \ge 1 \, ,\quad \pi^{ss} (z) \, W_j \, = \, \begin{cases}
\M(z)^{j-1} \, \pi^{ss}(z) \, W_1 \, ,& \text{\rm if $1 \, \le \, j \, \le \, j_0$,} \\
\M(z)^{j-1} \, \pi^{ss}(z) \, W_1 \, - \, a_p^{-1} \, \M(z)^{j-j_0-1} \, \pi^{ss}(z) \, {\bf e} \, ,& \text{\rm if $j \, > \, j_0$,}
\end{cases}
$$
and the coordinate of $(W_j)_{j \ge 1}$ along the eigenvector $E(z)$ is defined by the formula:
$$
\forall \, j \ge 1 \, ,\quad \mu_j \, = \, \begin{cases}
\kappa (z)^{j-1} \, \mu_1 \, ,& \text{\rm if $1 \, \le \, j \, \le \, j_0$,} \\
\kappa (z)^{j-1} \, \mu_1 \, - \, a_p^{-1} \, \kappa (z)^{j-j_0-1} \, \mu({\bf e}) \, ,& \text{\rm if $j \, > \, j_0$.}
\end{cases}
$$
As for the unstable components, we observe that for each couple of integers $j,j_0$, the above components of $W_j$ extend holomorphically 
to the whole ball $B_\varepsilon(1)$ since the spectral projectors of $\M(z)$ along $\E^{ss}(z)$ and $\text{\rm Span } E(z)$ do so. We thus 
consider from now on the holomorphic extension of $W_j$ for $z \in B_\varepsilon (1)$ and collect the three pieces of the vector $W_j$. For 
$1 \le j \le j_0$, we have:
$$
W_j \, = \, a_p^{-1} \, \M(z)^{-(j_0+1-j)} \, \pi^u(z) \, {\bf e} \, + \, \M(z)^{j-1} \, \pi^{ss}(z) \, W_1 \, + \, \kappa (z)^{j-1} \, \mu_1 \, E(z) \, ,
$$
which satisfies the bound:
$$
\Big| W_j \Big| \, \le \, C \, \exp (-c\, (j_0 \, - \, j)) \, + \, C \, \exp (-c\, (j_0 \, + \, j)) \, + \, C \, |\kappa (z)|^{\, j} \, \exp (- \, c \, j_0) \, ,
$$
for some constants $C>0$ and $c>0$ that are uniform with respect to $z$. Since $\kappa(1)=1$, we can always assume that there holds 
$|\kappa(z)| \le \exp c$ on the ball $B_\varepsilon (1)$, and we are then left with the estimate:
$$
\Big| W_j \Big| \, \le \, C \, \exp (-c\, (j_0 \, - \, j)) \, ,
$$
as claimed in the statement of Lemma \ref{lem:5}.

It remains to examine the case $j>j_0$ for which we have the decomposition:
$$
W_j \, = \, \M(z)^{j-1} \, \pi^{ss}(z) \, W_1 \, - \, a_p^{-1} \, \M(z)^{j-j_0-1} \, \pi^{ss}(z) \, {\bf e} \, + \, \kappa (z)^{j-1} \, \mu_1 \, E(z) \, 
- \, a_p^{-1} \, \kappa (z)^{j-j_0-1} \, \mu({\bf e}) \, E(z) \, ,
$$
and we can thus derive the bound:
$$
\Big| W_j \Big| \, \le \, C \, \exp (-c\, (j_0 \, + \, j)) \, + \, C \, \exp (-c\, (j \, - \, j_0)) \, + \, C \, |\kappa (z)|^{\, j} \, \exp (- \, c \, j_0) 
\, + \, C \, |\kappa (z)|^{|j \, - \, j_0|} \, .
$$
Since we can always assume that the ball $B_\varepsilon(1)$ is so small that $|\kappa(z)|$ takes its values within the interval $[\exp (-c),\exp c]$, 
it appears that the largest term on the above right hand side is the last term, which completes the proof of Lemma \ref{lem:5}.
\end{proof}

Let us observe that we can extend holomorphically each scalar component $G_z(j,j_0)$ but that does not mean that we can extend holomorphically 
$G_z(\cdot,j_0)$ in $\mathcal{H}$. As a matter of fact, the eigenvalue $\kappa(z)$ starts contributing to the unstable subspace of $\M(z)$ as $z$ 
(close to $1$) crosses the curve \eqref{curve-spectrum}. The holomorphic extension $G_z(\cdot,j_0)$ then ceases to be in $\ell^2$ for it has an 
exponentially growing mode in $j$. The last case to examine is that of the neighborhood of each eigenvalue $\underline{z}_k$.

\begin{lemma}[Bounds close to the eigenvalues]
\label{lem:6}
For any eigenvalue $\underline{z}_k \in \cercle$ of $\mathcal{T}$, there exists an open ball $B_\varepsilon(\underline{z}_k)$ centered at $\underline{z}_k$, 
there exists a sequence $(\underline{w}_k(j,j_0))_{j,j_0 \ge 1}$ with $\underline{w}_k(\cdot,j_0) \in \mathcal{H}$ for all $j_0 \ge 1$, and there exist two 
constants $\underline{C}_k>0$ and $\underline{c}_k>0$ such that for any couple of integers $(j,j_0)$, the component $G_z(j,j_0)$ defined on 
$B_\varepsilon(\underline{z}_k) \setminus \{ \underline{z}_k \}$ is such that:
$$
R_z(j,j_0) \, := \, G_z(j,j_0) \, - \, \dfrac{\underline{w}_k(j,j_0)}{z-\underline{z}_k} \, ,
$$
extends holomorphically to the whole ball $B_\varepsilon(\underline{z}_k)$ with respect to $z$, and the holomorphic extension satisfies the bound:
$$
\forall \, z \in B_\varepsilon(\underline{z}_k) \, ,\quad \big| \, R_z(j,j_0) \, \big| \, \le \, \underline{C}_k \, \exp \, \big( -\underline{c}_k \, |j \, - \, j_0| \, \big) \, .
$$
Moreover, the sequence $(\underline{w}_k(j,j_0))_{j,j_0 \ge 1}$ satisfies the pointwise bound:
$$
\forall \, j,j_0 \ge 1 \, ,\quad \big| \, \underline{w}_k(j,j_0) \, \big| \, \le \, \underline{C}_k \, \exp \, \big( \, - \, \underline{c}_k \, (j \, + \, j_0) \, \big) \, .
$$
\end{lemma}

\begin{proof}
Many ingredients for the proof of Lemma \ref{lem:6} are already available in the proof of Lemma \ref{lem:4}. Namely, let us consider an eigenvalue 
$\underline{z}_k \in \cercle$ of $\mathcal{T}$. Since $\underline{z}_k \in \mathcal{C}^c$, the matrix $\M(z)$ enjoys the hyperbolic dichotomy between 
its stable and unstable eigenvalues in the neighborhood of $\underline{z}_k$. Moreover, for a sufficiently small radius $\varepsilon>0$, the pointed ball 
$B_\varepsilon(\underline{z}_k) \setminus \{ \underline{z}_k \}$ lies in the resolvant set of $\mathcal{T}$. In particular, for any $z \in B_\varepsilon 
(\underline{z}_k) \setminus \{ \underline{z}_k \}$, the spatial Green's function is obtained by selecting the appropriate scalar component of the vector 
sequence $(W_j)_{j \ge 1}$ defined by:
\begin{equation}
\label{lem6-1}
\forall \, j \ge 1 \, ,\quad \pi^u (z) \, W_j \, := \, \begin{cases}
0  \, ,& \text{\rm if $j \, > \, j_0$,} \\
a_p^{-1} \, \M(z)^{-(j_0+1-j)} \, \pi^u(z) \, {\bf e} \, ,& \text{\rm if $1 \, \le \, j \, \le \, j_0$,}
\end{cases}
\end{equation}
and:
\begin{equation}
\label{lem6-2}
\forall \, j \ge 1 \, ,\quad \pi^s (z) \, W_j \, := \, \begin{cases}
\M(z)^{j-1} \, \pi^s(z) \, W_1 \, ,& \text{\rm if $1 \, \le \, j \, \le \, j_0$,} \\
\M(z)^{j-1} \, \pi^s(z) \, W_1 -a_p^{-1} \, \M(z)^{j-j_0-1} \, \pi^s(z) \, {\bf e} \, ,& \text{\rm if $j \, > \, j_0$,}
\end{cases}
\end{equation}
where the vector $\pi^s (z) \, W_1 \in \E^s(z)$ is defined by (see \eqref{composantestable1}):
\begin{equation}
\label{lem6-3}
\pi^s(z) \, W_1 \, := -\, a_p^{-1} \, \Big( \mathcal{B}|_{\E^s(z)} \Big)^{-1} \, \mathcal{B} \, \M(z)^{-j_0} \, \pi^u(z) \, {\bf e} \, .
\end{equation}
(Here we use the fact that for every $z$ in the pointed ball $B_\varepsilon(\underline{z}_k) \setminus \{ \underline{z}_k \}$, the linear map 
$\mathcal{B}|_{\E^s(z)}$ is an isomorphism.)

The unstable component $\pi^u (z) \, W_j$ in \eqref{lem6-1} obviously extends holomorphically to the whole ball $B_\varepsilon(\underline{z}_k)$ 
and the estimate \eqref{borneslem4-1} shows that this contribution to the remainder term $R_z(j,j_0)$ satisfies the desired uniform exponential 
bound with respect to $z$. We thus focus from now on on the stable components defined by \eqref{lem6-2}, \eqref{lem6-3}. We first observe that, 
as in the unstable component \eqref{lem6-1}, the contribution:
$$
-a_p^{-1} \, \M(z)^{j-j_0-1} \, \pi^s(z) \, {\bf e}
$$
appearing in the definition \eqref{lem6-2} for $j>j_0$ also extends holomorphically to the ball $B_\varepsilon(\underline{z}_k)$ and contributes to 
the remainder term $R_z(j,j_0)$ with an $O(\exp(-c\, |j-j_0|))$ term. We thus focus on the sequence:
$$
\Big( \, \M(z)^{j-1} \, \pi^s(z) \, W_1 \, \Big)_{j \ge 1} \, ,
$$
where the vector $\pi^s(z) \, W_1 \in \E^s(z)$ is defined by \eqref{lem6-3} for $z \in B_\varepsilon(\underline{z}_k) \setminus \{ \underline{z}_k \}$. 
The singularity in the Green's function comes from the fact that $\mathcal{B}|_{\E^s(\underline{z}_k)}$ is no longer an isomorphism. We now make 
this singularity explicit.

We pick a basis $e_1(z),\dots,e_r(z)$ of the stable subspace $\E^s(z)$ that depends holomorphically on $z$ near $\underline{z}_k$. Since the 
Lopatinskii determinant factorizes as:
$$
\Delta (z) \, = \, (z \, - \, \underline{z}_k) \, \vartheta(z) \, ,
$$
where $\vartheta$ is a holomorphic function that does not vanish at $\underline{z}_k$, we can therefore write:
$$
\begin{pmatrix} 
\mathcal{B} \, e_1(z) & \cdots & \mathcal{B} \, e_r(z) \end{pmatrix}^{-1} \, = \, \dfrac{1}{z \, - \, \underline{z}_k} \, D(z) \, ,
$$
where $D(z)$ is a matrix in $\mathcal{M}_r(\C)$ that depends holomorphically on $z$ near $\underline{z}_k$. We then define the vector:
\begin{equation}
\label{lem6-4}
\underline{\mathcal{W}}(j_0) \, := \, \begin{pmatrix} 
e_1(\underline{z}_k) & \cdots & e_r(\underline{z}_k) \end{pmatrix} \, D(\underline{z}_k) \, \Big( 
 -\, a_p^{-1} \, \mathcal{B} \, \M(\underline{z}_k)^{-j_0} \, \pi^u(\underline{z}_k) \, {\bf e} \Big) \, ,
\end{equation}
which satisfies the bound:
\begin{equation}
\label{borneslem6-1}
\big| \, \underline{\mathcal{W}}(j_0) \, \big| \, \le \, C \, \exp( - \, c \, j_0) \, ,
\end{equation}
for some positive constants $C$ and $c$, uniformly with respect to $j_0 \ge 1$. Moreover, since we have the relation:
$$
\mathcal{B} \, \begin{pmatrix} 
e_1(\underline{z}_k) & \cdots & e_r(\underline{z}_k) \end{pmatrix} \, D(\underline{z}_k) \, = \, 0 \, ,
$$
the vector $\underline{\mathcal{W}}(j_0)$ belongs to $\text{\rm Ker } \mathcal{B} \, \cap \, \E^s (\underline{z}_k)$. Hence, by selecting the appropriate 
coordinate, the geometric sequence (which is valued in $\C^{p+r}$):
$$
\Big( \, \M(\underline{z}_k)^{j-1} \, \underline{\mathcal{W}}(j_0) \, \Big)_{j \ge 1} \, ,
$$
provides with a scalar sequence $(\underline{w}_k(j,j_0))_{j,j_0 \ge 1}$ with $\underline{w}_k(\cdot,j_0) \in \mathcal{H}$ for all $j_0 \ge 1$, and 
that satisfies the bound:
$$
\forall \, j,j_0 \ge 1 \, ,\quad \big| \, \underline{w}_k(j,j_0) \, \big| \, \le \, \underline{C}_k \, \exp \, \big( -\underline{c}_k \, (j \, + \, j_0) \, \big) \, ,
$$
as stated in Lemma \ref{lem:6}. It thus only remains to show that the remainder term:
\begin{equation}
\label{lem6-5}
\mathcal{R}_z(j,j_0) \, := \, \M(z)^{j-1} \, \pi^s(z) \, W_1 \, - \, \dfrac{\M(\underline{z}_k)^{j-1} \, \underline{\mathcal{W}}(j_0)}{z \, - \, \underline{z}_k}
\end{equation}
extends holomorphically to $B_\varepsilon(\underline{z}_k)$ and satisfies a suitable exponential bound.

We decompose the vector $\pi^s(z) \, W_1$ in \eqref{lem6-3} along the basis $e_1(z),\dots,e_r(z)$ of the stable subspace $\E^s(z)$ and write:
$$
\pi^s(z) \, W_1 \, = \, \dfrac{1}{z \, - \, \underline{z}_k} \, \begin{pmatrix} 
e_1(z) & \cdots & e_r(z) \end{pmatrix} \, D(z) \, \Big( - \, a_p^{-1} \, \mathcal{B} \, \M(z)^{-j_0} \, \pi^u(z) \, {\bf e} \Big) \, .
$$
Using the definitions \eqref{lem6-4} and \eqref{lem6-5}, we can decompose the remainder $\mathcal{R}_z(j,j_0)$ as follows:
\begin{align*}
\mathcal{R}_z(j,j_0) \, =& \, \dfrac{1}{z \, - \, \underline{z}_k} \, \Big( \, (\M(z) \, \pi^s(z))^{j-1} \, - \, (\M(\underline{z}_k) \, \pi^s(\underline{z}_k))^{j-1} 
\, \Big) \, \underline{\mathcal{W}}(j_0) \\
& \, -\dfrac{a_p^{-1}}{z \, - \, \underline{z}_k} \, (\M(z) \, \pi^s(z))^{j-1} \, \Big( \, \begin{pmatrix} 
e_1(z) & \cdots & e_r(z) \end{pmatrix} \, D(z) \, \mathcal{B} \, (\M(z) \, \pi^u(z))^{-j_0} \, {\bf e} \\
& \quad -\begin{pmatrix} e_1(\underline{z}_k) & \cdots & e_r(\underline{z}_k) \end{pmatrix} \, 
D(\underline{z}_k) \, \mathcal{B} \, (\M(\underline{z}_k) \, \pi^u(\underline{z}_k))^{-j_0} \, {\bf e} \Big) \, .
\end{align*}
Both terms (the first line, and the difference between the second and third lines) in the above decomposition are dealt with by applying the following 
result combined with the hyperbolic dichotomy of $\M(z)$ near $\underline{z}_k$.

\begin{lemma}
\label{lem-technique}
Let $M$ be a holomorphic function on the open ball $B_\delta(0)$ with values in $\mathcal{M}_N(\C)$ for some $\delta>0$ and integer $N$, that 
satisfies:
$$
\exists \, C \, > \, 0 \, ,\quad \exists \, r \in (0,1) \, ,\quad \forall \, j \in \N \, ,\quad \forall \, z \in B_\delta(0) \, ,\quad 
\| \, M(z)^{\, j} \, \| \, \le \, C \, r^{\, j} \, .
$$
Then up to diminishing $\delta$ and for some possibly new constants $C>0$ and $r \in (0,1)$, there holds:
$$
\forall \, j \in \N \, ,\quad \forall \, z_1,z_2 \in B_\delta(0) \, ,\quad \| \, M(z_1)^{\, j} \, - \, M(z_2)^{\, j} \, \| \, \le \, C \, |z_1 \, - \, z_2| \, r^{\, j} \, .
$$
\end{lemma}

Applying Lemma \ref{lem-technique} to the above decomposition of $\mathcal{R}_z(j,j_0)$, and using the exponential decay of the sequences 
$(\M(z)^{\, j} \, \pi^s(z))_{j \in \N}$ and $(\M(z)^{-j} \, \pi^u(z))_{j \in \N}$, we get the bound:
$$
\big| \, \mathcal{R}_z(j,j_0) \, \big| \, \le \, C \, \exp \big( \, - \, \underline{c}_k \, (j \, + \, j_0) \, \big) \, ,
$$
which means that $\mathcal{R}_z(j,j_0)$ remains bounded on $B_\varepsilon(\underline{z}_k) \setminus \{ \underline{z}_k \}$ and can therefore be 
extended holomorphically to the whole ball $B_\varepsilon(\underline{z}_k)$. The proof of Lemma \ref{lem:6} is complete.
\end{proof}

\begin{proof}[Proof of Lemma \ref{lem-technique}]
The argument is a mere application of the Taylor formula. Let us recall that the differential of the mapping:
$$
\Psi_j \quad : \quad A \in \mathcal{M}_N(\C) \quad \longmapsto \quad A^{\, j} \, ,
$$
is given by:
$$
{\rm d}\Psi_j (A) \cdot B \, = \, \sum_{\ell=0}^{j-1} \, A^{\, \ell} \, B \, A^{\, j-1-\ell} \, ,
$$
so we have:
$$
M(z_1)^{\, j} \, - \, M(z_2)^{\, j} \, = \, (z_1 \, - \, z_2) \, \int_0^1 \, \sum_{\ell=0}^{j-1} \, M(z_2 \, +t \, (z_1-z_2))^{\, \ell} \, M'(z_2 \, +t \, (z_1-z_2)) \, 
M(z_2 \, +t \, (z_1-z_2))^{\, j-1-\ell} \, {\rm d}t \, .
$$
The result follows by using a uniform bound for the first derivative $M'$, up to diminishing $\delta$, and using the exponential decay of the sequence 
$(j \, r^{\, j})_{j \in \N}$.
\end{proof}

\subsection{Summary}

Collecting the results of Lemma \ref{lem:4}, Lemma \ref{lem:5} and Lemma \ref{lem:6}, we can obtain the following bound for the spatial Green's function 
away from the spectrum of $\mathcal{T}$.

\begin{corollary}
\label{cor1}
There exist a radius $\varepsilon>0$, some width $\eta_\varepsilon>0$ and two constants $C_0>0$, $c_0>0$ such that, for all $z$ in the set:
$$
\Big\{ \zeta \in \C \, / \, {\rm e}^{- \, \eta_\varepsilon} \, < \, |\zeta| \, \le \, {\rm e}^\pi \, \Big\} \, \setminus \, \left( \, 
\bigcup_{k=1}^K B_\varepsilon(\underline{z}_k) \, \cup \, B_\varepsilon(1) \, \right) \, ,
$$
and for all $j_0 \ge 1$, the Green's function $G_z(\cdot,j_0) \in \mathcal{H}$ solution to \eqref{defGzj} satisfies the pointwise bound:
$$
\forall \, j \ge 1 \, ,\quad \big| \, G_z(j,j_0) \, \big| \, \le \, C_0 \, \exp \big( - \, c_0 \, |j \, - \, j_0| \, \big) \, .
$$
Moreover, for $z$ inside the ball $B_\varepsilon(1)$, the Green's function component $G_z(j,j_0)$ depends holomorphically on $z$ and 
satisfies the bound given in Lemma \ref{lem:5}, and for $k=1,\dots,K$ and $z$ in the pointed ball $B_\varepsilon(\underline{z}_k) \setminus 
\{ \underline{z}_k \}$, $G_z(j,j_0)$ has a simple pole at $\underline{z}_k$ with the behavior stated in Lemma \ref{lem:6}.
\end{corollary}

\section{Temporal Green's function and proof of Theorem~\ref{thm1}}
\label{sect:proof}

The starting point of the analysis is to use inverse Laplace transform formula to  express the temporal Green's function $\G^{\, n}(\cdot,j_0) := 
\mathcal{T}^{\, n} \, \boldsymbol{\delta}_{j_0}$ as the following contour integral:
\begin{equation}
\forall \, n \in \N^* \, ,\quad \forall \, j \geq 1 \, ,\quad \forall \, j_0 \geq 1 \, ,\quad \G^{\, n}(j,j_0) \, = \, 
\left( \, \mathcal{T}^{\, n} \, \boldsymbol{\delta}_{j_0} \, \right)_j \, = \, \dfrac{1}{2 \, \pi \, \mathbf{i}} \, \int_{\widetilde\Gamma} \, z^n \, G_z(j,j_0) \, \md z \, ,
\label{lapgreenz}
\end{equation}
where $\widetilde{\Gamma}$ is a closed curve in the complex plane surrounding the unit disk $\D$ lying in the resolvent set of $\mathcal{T}$ and 
$G_z(\cdot,j_0) \in \mathcal{H}$ is the spatial Green's function defined in \eqref{defGzj}.

Following our recent work \cite{CF1}, the idea will be to deform $\widetilde{\Gamma}$ in order to obtain sharp pointwise estimates on the temporal 
Green's function using our pointwise estimates on the spatial Green's function summarized in Corollary~\ref{cor1} above. To do so, we first change 
variable in \eqref{lapgreenz}, by setting $z=\exp(\tau)$, such that we get
\bqq
\G^{\, n}(j,j_0) \, = \, \dfrac{1}{2 \, \pi \, \mathbf{i}} \, \int_{\Gamma} \, {\rm e}^{n \, \tau} \, \mathbf{G}_\tau(j,j_0) \, \md \tau \, ,
\label{lapgreentau}
\eqq
where without loss of generality $\Gamma=\left\{ s+\mathbf{i} \, \ell ~|~ \ell \in[-\pi,\pi]\right\}$ for some (and actually any) $s>0$, and $\mathbf{G}_\tau 
(\cdot,j_0) \in \mathcal{H}$ is given by
\bqs
\forall \, \tau \in \Gamma \, ,\quad \forall \, j \geq 1 \, ,\quad \forall \, j_0 \geq 1 \, ,\quad 
\mathbf{G}_\tau(j,j_0) \, := \, G_{{\rm e}^{\tau}}(j,j_0) \, {\rm e}^{\tau} \, .
\eqs

It is already important to remark that as $\T$ is a recurrence operator with finite stencil, for each $n\geq 1$, there holds
\bqs
\G^{\, n} (j,j_0) \, = \, 0 \, , \, \text{ for } \, j \, - \, j_0 \, > \, r \, n \, \text{ or } \, j \, - \, j_0 \, < \, - \, p \, n \, .
\eqs
As a consequence, throughout this section, we assume that $j$, $j_0$ and $n$ satisfy
\bqs
- \, p \, n \, \leq \, j \, - \, j_0 \, \leq \, r \, n \, .
\eqs

The very first step in the analysis of the temporal Green's function defined in \eqref{lapgreentau} is to translate the pointwise estimates from 
Corollary~\ref{cor1} for the spatial Green's function $G_z(j,j_0)$ to pointwise estimates for $\mathbf{G}_\tau(j,j_0)$. We let be $\underline{\tau}_k 
= \mathbf{i} \, \underline{\theta}_k := \log(\underline{z}_k)$ for $\underline{\theta}_k \in [-\pi,\pi] \setminus \{ 0 \}$ for each $k=1,\cdots,K$. Finally, 
we also set $\alpha:=\lambda \, a>0$. As in \cite{CF1}, the temporal Green's function is expected to have a leading order contribution concentrated 
near $j-j_0 \sim \alpha \, n$. An important feature of the situation we deal here with is that the temporal Green's function should also incorporate the 
contribution of the eigenvalues $\underline{z}_k$. These contributions will not decay with respect to $n$ since the $\underline{z}_k$'s have modulus 
$1$.

\begin{lemma}
\label{lem8}
There exist a radius $\varepsilon>0$, some width $\eta_\varepsilon>0$ and constants $0<\beta_*<\beta<\beta^*$ and $C>0$, $c>0$ such that, for all 
$z$ in the set:
$$
\Omega_\varepsilon := \Big\{ \tau \in \C \, | \, - \, \eta_\varepsilon \, < \, \Re(\tau) \, \le \, \pi \, \Big\} \, \setminus \, \left( \, 
\bigcup_{k=1}^K B_\varepsilon(\mathbf{i} \, \underline{\theta}_k) \, \cup \, B_\varepsilon(0) \, \right) \, ,
$$
and for all $j_0 \ge 1$, the Green's function $\mathbf{G}_\tau(\cdot,j_0)\in\mathcal{H}$ satisfies the pointwise bound:
$$
\forall \, j,j_0 \ge 1 \, ,\quad \big| \, \mathbf{G}_\tau(j,j_0) \, \big| \, \le \, C \, \exp \big( - \, c \, |j \, - \, j_0| \, \big) \, .
$$
Moreover, for $\tau$ inside the ball $B_\varepsilon(0)$, the Green's function component $\mathbf{G}_\tau(j,j_0)$ depends holomorphically on $\tau$ and 
satisfies the bound 
\bqs
\forall \, \tau \in B_\varepsilon(0) \, ,\quad \forall \, j,j_0 \ge 1 \, ,\quad \big| \, \mathbf{G}_\tau(j,j_0) \, \big| \, \le \, \begin{cases}
C \, \exp \big( -c \, |j-j_0| \, \big) \, ,& \text{\rm if $1 \, \le \, j \, \le \, j_0 $,} \\
C \, \exp \big( \, |j-j_0| \, \Re(\varpi(\tau)) \, \big) \, ,& \text{\rm if $j \, > \, j_0$,}
\end{cases} 
\eqs
with
\bqs
\varpi(\tau) \, = \, - \, \dfrac{1}{\alpha} \, \tau \, + \, (-1)^{\mu+1} \, \dfrac{\beta}{\alpha^{2 \, \mu+1}} \, \tau^{2 \, \mu} \, + \, 
O \left( \, |\tau|^{2 \, \mu+1} \, \right) \, ,\quad \, \forall \, \tau \in B_\varepsilon(0) \, ,
\eqs
together with
\bqs
\Re(\varpi(\tau)) \, \leq \, - \, \dfrac{1}{\alpha} \, \Re(\tau) \, + \,  \, \dfrac{\beta^*}{\alpha^{2 \, \mu+1}} \, \Re(\tau)^{2 \, \mu} \, -  \, 
\dfrac{\beta_*}{\alpha^{2 \, \mu+1}} \, \Im(\tau)^{2 \, \mu} \, ,\quad \, \forall \, \tau \in B_\varepsilon(0) \, .
\eqs
At last, for any $k=1,\dots,K$ and $\tau$ in the pointed ball $B_\varepsilon(\mathbf{i} \, \underline{\theta}_k) \setminus \{\mathbf{i} \, \underline{\theta}_k \}$, 
$\mathbf{G}_\tau(j,j_0)$ has a simple pole at $\mathbf{i} \, \underline{\theta}_k$ with the following behavior. There exists a sequence $(\underline{{\bf w}}_k 
(j,j_0))_{j,j_0 \ge 1}$ with $\underline{{\bf w}}_k(\cdot,j_0) \in \mathcal{H}$ for all $j_0 \ge 1$, such that:
$$
\forall \, j,j_0 \ge 1 \, ,\quad 
\mathbf{R}_\tau(j,j_0) \, := \, \mathbf{G}_\tau(j,j_0) \, - \, \dfrac{\underline{{\bf w}}_k(j,j_0)}{\tau-\mathbf{i} \, \underline{\theta}_k} \, ,
$$
extends holomorphically to the whole ball $B_\varepsilon(\mathbf{i} \, \underline{\theta}_k)$ with respect to $\tau$, and the holomorphic extension 
satisfies the bound:
$$
\forall \, \tau \in B_\varepsilon(\mathbf{i} \, \underline{\theta}_k) \, ,\quad \big| \, \mathbf{R}_\tau(j,j_0) \, \big| \, \le \, C \, \exp \, \big( -c \, |j \, - \, j_0| \, \big) \, .
$$
Moreover, the sequence $(\underline{{\bf w}}_k(j,j_0))_{j,j_0 \ge 1}$ satisfies the pointwise bound:
\begin{equation}
\label{bornesprofils}
\forall \, j,j_0 \ge 1 \, ,\quad \big| \, \underline{{\bf w}}_k(j,j_0) \, \big| \, \le \, C \, \exp \, \big( \, - \, c \, (j \, + \, j_0) \, \big) \, .
\end{equation}
\end{lemma}

\begin{proof}
The proof simply relies on writing $\kappa(z)=\exp(\omega(z))$ and using $z=\exp(\tau)$, such that after identification we have $\varpi(\tau):=\omega(\exp(\tau))$. 
Next, using our assumption \eqref{hyp:stabilite2}, we obtain the desired expansion for $\varpi(\tau)$ near $\tau=0$. From this expansion, we get
\begin{align*}
\Re(\varpi(\tau)) \, &= \, - \, \dfrac{1}{\alpha} \, \Re(\tau) \, - \, \dfrac{\beta}{\alpha^{2 \, \mu+1}} \, \Im(\tau)^{2 \, \mu} 
\, - \, \dfrac{\beta}{\alpha^{2 \, \mu+1}} \, (-1)^\mu \, \Re(\tau)^{2 \, \mu} \\
&~~~-\frac{\beta}{\alpha^{2 \, \mu+1}} \, \sum_{m=1}^{\mu-1} \, (-1)^m \, 
\left( \begin{matrix}
2 \, \mu \\
2 \, m \end{matrix} \right) \, \Re(\tau)^{2 \, m} \, \Im(\tau)^{2 \, (\mu-m)} \, + \, O(|\tau|^{2 \, \mu+1}) \, .
\end{align*}
for all $\tau \in B_\varepsilon(0)$. We crucially note that the term $\Im(\tau)^{2 \, \mu}$ comes with a negative sign such that both $O(|\tau|^{2 \, \mu+1})$ 
and each term of the sum, using Young's inequality, can be absorbed and we arrive at the desired estimate for two uniform constants $0<\beta_*<\beta<\beta^*$. 
The remainder of the proof is a simple transposition of Lemma~\ref{lem:4},~\ref{lem:5} and \ref{lem:6} in the new variable $\tau$.
\end{proof}

\begin{figure}[t!]
  \centering
 \includegraphics[width=.48\textwidth]{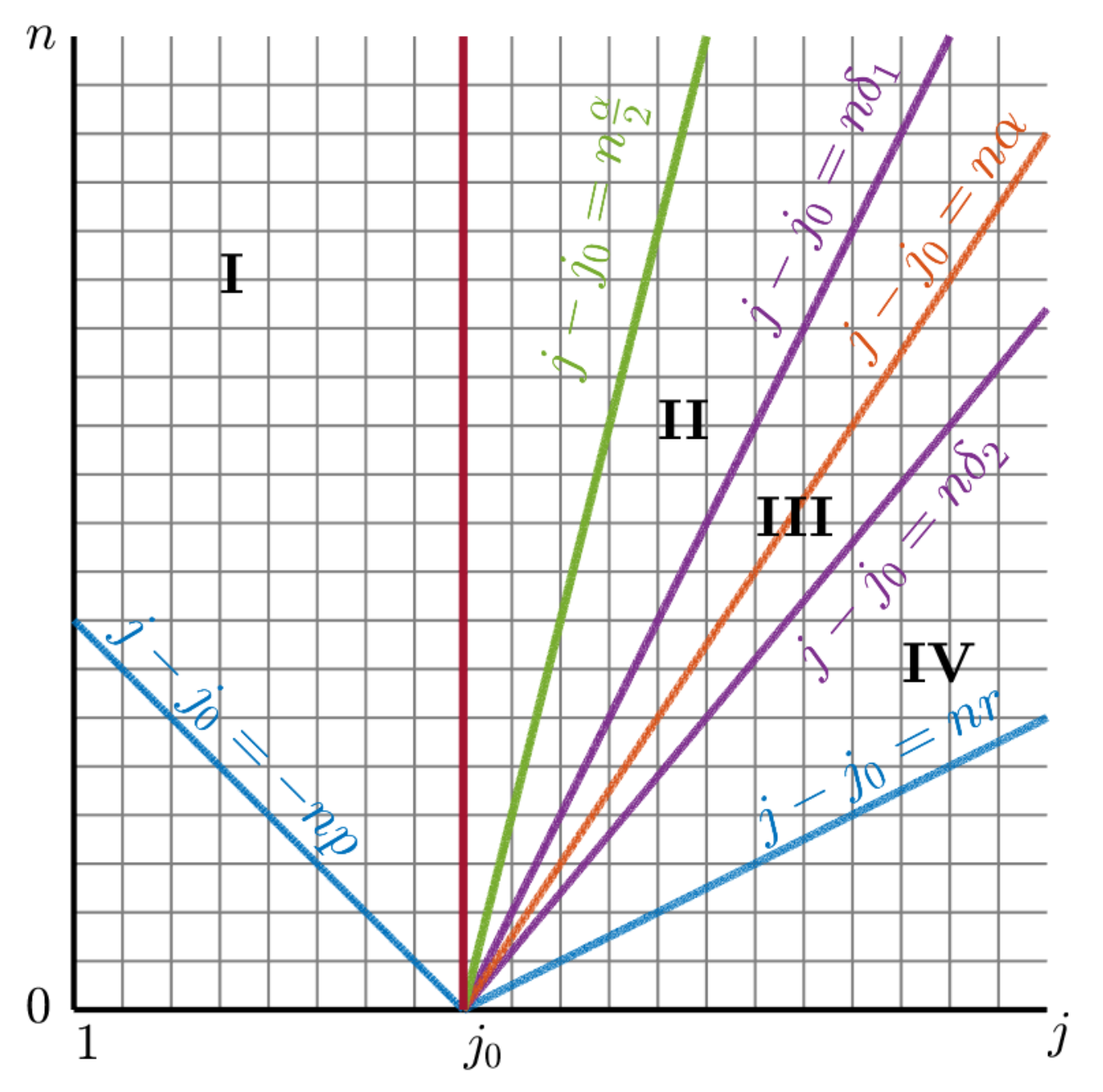}
  \caption{Schematic illustration of the different domains in the $(j,n)$ plane used in the analysis. In each domain, we use a different contour integral in 
  \eqref{lapgreentau}. In domain \textbf{I}, that is for $- \, n \, p\leq j-j_0\leq n \frac{\alpha}{2}$, we can push the contour of integration $\Gamma$ to $\Re(\tau) 
  =-\eta$ for some well chosen $\eta>0$. For values in domains \textbf{II}, \textbf{III} and \textbf{IV},  we can also push the contour of integration $\Gamma$ 
  to $\Re(\tau)=-\eta$ but this time we have to use a  ``parabolic'' contour near the origin. We refer to Figure~\ref{fig:contour} for an illustration of such 
  contours in domain \textbf{III}. Note that below the lines $j-j_0=- n \, p$ and $j-j_0=n \, r$ (blue) the Green's function $\G^{\, n}(j,j_0)$ vanishes.}
  \label{fig:spacetime}
\end{figure}

With the notations introduced in the above Lemma, we can summarize in the following proposition the results that we will prove in this section. Why Proposition 
\ref{prop1} is sufficient to get the result of Theorem \ref{thm1} is explained at the end of this Section.

\begin{proposition}
\label{prop1}
There exist two constants $C>0$ and $\omega>0$ such that for any $n \geq 1$ the temporal Green's function $\G^{\, n}$ satisfies the pointwise estimate
\bqs
\forall \, j,j_0 \ge 1 \, ,\quad 
\left| \, \mathcal{G}^{\, n}(j,j_0) \, - \, \sum_{k=1}^K \, \underline{{\bf w}}_k(j,j_0) \, {\rm e}^{\, n \, \mathbf{i} \, \underline{\theta}_k} \, \right| \, \le \, 
\dfrac{C}{n^{\frac{1}{2 \, \mu}}} \, \exp \left( -\, \omega \, \left( \dfrac{|j-j_0- \alpha \, n|}{n^{\frac{1}{2 \, \mu}}} \right)^{\frac{2 \, \mu}{2 \, \mu-1}}\right) \, .
\eqs
\end{proposition}

From now on, we fix $0<\eta<\eta_\varepsilon$ such that  the segment $\left\{-\eta +\mathbf{i} \, \ell~|~\ell\in[-\pi,\pi]\right\}$ intersects $\partial B_\varepsilon(0)$ 
outside the curve \eqref{curve-spectrum} of essential spectrum of $\mathcal{T}$ near the origin. We are going to distinguish several cases depending on the 
relative position between $j-j_0$ and $n$, as sketched in Figure~\ref{fig:spacetime}. Formally, we will use different contours of integration in \eqref{lapgreentau} 
depending  if $j-j_0$ is near $n \, \alpha$ or away from  $n \, \alpha$. Indeed, when $j-j_0 \sim n \, \alpha$, we expect to have Gaussian-like bounds coming 
from the contribution in $B_\varepsilon(0)$ near the origin where the essential spectrum of $\T$ touches the imaginary axis. In that case, we will use contours 
similar to \cite{godillon,CF1} and that were already introduced in the continuous setting in \cite{ZH98}. Let us note that unlike in \cite{CF1}, we have isolated 
poles on the imaginary axis given by the $\underline{\tau}_k=\mathbf{i} \, \underline{\theta}_k$, $k=1,\cdots,K$, whose contributions in \eqref{lapgreentau} 
will be handled via Cauchy's formula and the residue theorem. We thus divide the analysis into a medium range, that is for those values of $j-j_0$ away from 
$n \, \alpha$, and short range when $j-j_0$ is near $n \, \alpha$. More specifically, we decompose our domain as
\begin{itemize}
\item[$\bullet$] Medium range: $- \, n \, p \leq j-j_0 < n \, \frac{\alpha}{2} \, $;
\item[$\bullet$] Short range: $n \, \frac{\alpha}{2} \leq j-j_0 \leq n \, r \, $;
\end{itemize}
where we recall that $\alpha=\lambda \, a>0$ from our consistency condition.

\subsection{Medium range}

In this section, we consider the medium range where $- \, n \, p\leq j-j_0< n \frac{\alpha}{2}$. In order to simplify the presentation, we first treat the case where 
$- \, n \, p \leq j-j_0 \leq 0$ and then consider the range $1 \leq j-j_0 < n \, \frac{\alpha}{2} $.

\begin{lemma}
\label{lem9}
There exist constants $C>0$ and $c>0$, such that for all integers $j,j_0,n$ satisfying $- \, n \, p\leq j-j_0\leq 0$, the temporal Green's function satisfies
\bqs
\left| \, \G^{\, n}(j,j_0) \, - \, \sum_{k=1}^K \, \underline{{\bf w}}_k(j,j_0) \, {\rm e}^{\, n \, \mathbf{i} \, \underline{\theta}_k} \, \right| \, \leq \, 
C \, {\rm e}^{- \, n \, \eta \, - \, c \, |j \, - \, j_0|} \, .
\eqs
\end{lemma}

\begin{proof}
We first recall that 
$$
\G^{\, n}(j,j_0) \, = \, \dfrac{1}{2 \, \pi \, \mathbf{i}} \, \int_{\Gamma} \, {\rm e}^{\, n \, \tau} \, \mathbf{G}_\tau(j,j_0) \, \md \tau \, ,
$$
with $\Gamma=\left\{ s+\mathbf{i} \, \ell ~|~ \ell \in [-\pi,\pi] \right\}$ for any $0<s \leq \pi$. Next, we denote $\Gamma_{-\eta}=\left\{ -\eta+\mathbf{i} \, \ell ~|~ \ell 
\in [-\pi,\pi] \right\}$. Using the residue theorem, we obtain that
\bqs
\dfrac{1}{2 \, \pi \, \mathbf{i}} \, \int_{\Gamma} \, {\rm e}^{\, n \, \tau} \, \mathbf{G}_\tau(j,j_0) \, \md \tau \, = \, 
\underbrace{ \dfrac{1}{2 \, \pi \, \mathbf{i}} \, \int_{\Gamma_{-\eta}} \, {\rm e}^{\, n \, \tau} \, \mathbf{G}_\tau(j,j_0) \, \md \tau}_{:= \, \widetilde{\G}^{\, n}(j,j_0)} 
\, + \, \sum_{k=1}^K \, \mathrm{Res} \Big( \tau \mapsto {\rm e}^{\, n \, \tau} \, \mathbf{G}_\tau(j,j_0) , \underline{\tau}_k \Big) \, ,
\eqs
where we readily have that
\bqs
\sum_{k=1}^K \, \mathrm{Res} \Big( \tau \mapsto {\rm e}^{\, n \, \tau} \, \mathbf{G}_\tau(j,j_0) , \underline{\tau}_k \Big) \, = \, 
\sum_{k=1}^K \, \underline{{\bf w}}_k(j,j_0) \, {\rm e}^{\, n \, \mathbf{i} \, \underline{\theta}_k} \, ,
\eqs
from Lemma~\ref{lem8}. Here, and throughout, we use the fact that the integrals along $\left\{ -v \pm \mathbf{i} \, \pi ~|~ v \in [-\eta,s] \right\}$ compensate each 
other. Now, $\Gamma_{-\eta}$ intersects each ball $B_\varepsilon(\mathbf{i} \, \underline{\theta}_k)$ and we denote $\Gamma_{-\eta}^k := 
\Gamma_{-\eta} \cap B_\varepsilon(\mathbf{i} \, \underline{\theta}_k)$. Using once again Lemma~\ref{lem8}, we have for each $\tau \in \Gamma_{-\eta}^k$
\bqs
\left| \, \mathbf{G}_\tau(j,j_0) \, \right| \, \leq \, \left| \, \mathbf{R}_\tau(j,j_0) \, \right| \, + \, 
\left| \, \dfrac{\underline{{\bf w}}_k(j,j_0)}{\tau-\mathbf{i} \, \underline{\theta}_k} \, \right| 
\, \leq \, C \, {\rm e}^{- \, c \, |j-j_0|} \, + \, C_0 \, {\rm e}^{- \, c \, |j+j_0|} \, \leq \, C_1 \, {\rm e}^{- \, c \, |j-j_0|} \, ,
\eqs
for some positive constants $C_{0,1}>0$. Finally, we remark that for $- \, n \, p \leq j-j_0 \leq 0$, we have
\bqs
\forall \, \tau \in \Big\{ \omega \in \C \, | \, - \, \eta_\varepsilon \, < \, \Re(\omega) \, \le \, \pi \, \Big\} \, \setminus \, \left( \, 
\bigcup_{k=1}^K B_\varepsilon(\mathbf{i} \, \underline{\theta}_k) \, \right) \, , \, \quad \big| \, \mathbf{G}_\tau(j,j_0) \, \big| \, \leq \, C \, {\rm e}^{- \, c \, |j \, - \, j_0| } \, ,
\eqs
such that in fact, for all $\tau \in \Gamma_{-\eta}$ and $- \, n \, p \leq j-j_0 \leq 0$ we have the following bound
\bqs
\big| \, \mathbf{G}_\tau(j,j_0) \, \big| \, \leq \, C \, {\rm e}^{- \, c \, |j \, - \, j_0| } \, .
\eqs
The estimate on $\widetilde{\G}^{\, n}(j,j_0)$ easily follows and concludes the proof.
\end{proof}

Next, we consider the range $1 \leq j-j_0 < n \, \frac{\alpha}{2} $. This time, the spatial Green's function $\mathbf{G}_\tau(j,j_0)$ satisfies a different bound 
in $B_\varepsilon(0)$. Nevertheless, we can still obtain some strong decaying estimates which are summarized in the following lemma.

\begin{lemma}
\label{lem10}
There exists a constant $C>0$ such that for all integers $j,j_0,n$ satisfying $1 \leq j-j_0 < n \, \frac{\alpha}{2}$, the temporal Green's function satisfies
\bqs
\left| \, \G^{\, n}(j,j_0) \, - \, \sum_{k=1}^K \, \underline{{\bf w}}_k(j,j_0) \, {\rm e}^{\, n \, \mathbf{i} \, \underline{\theta}_k} \, \right| \, \leq \, 
C \, {\rm e}^{- \, n \, \frac{\eta}{4}} \, .
\eqs
\end{lemma}

\begin{proof}
The beginning of the proof follows similar lines as the ones in the proof of Lemma~\ref{lem9}. We deform the initial contour $\Gamma$ to $\Gamma_{-\eta}$, 
and using the residue theorem we get
\bqs
\G^{\, n}(j,j_0) \, = \, \sum_{k=1}^K \, \underline{{\bf w}}_k(j,j_0) \, {\rm e}^{\, n \, \mathbf{i} \, \underline{\theta}_k} \, + \, 
\dfrac{1}{2 \, \pi \, \mathbf{i}} \, \int_{\Gamma_{-\eta}} \, {\rm e}^{\, n \, \tau} \, \mathbf{G}_\tau(j,j_0) \, \md \tau \, .
\eqs
We denote by $\Gamma^{in}_{-\eta}$ and $\Gamma^{out}_{-\eta}$ the portions of $\Gamma_{-\eta}$ which lie either inside or outside  $B_\varepsilon(0)$. 
Note that the analysis along $\Gamma^{out}_{-\eta}$ is similar as in Lemma~\ref{lem9}, and we already get the estimate
\bqs
\left| \frac{1}{2 \, \pi \, \mathbf{i}} \, \int_{\Gamma_{-\eta}^{out}} \, {\rm e}^{\, n \, \tau} \, \mathbf{G}_\tau(j,j_0) \, \md \tau \right| 
\, \leq \, C \, {\rm e}^{- \, n \, \eta \, - \, c \, |j-j_0|} \, \leq \, C \, {\rm e}^{- \, n \, \frac{\eta}{4}} \, .
\eqs
Along $\Gamma^{in}_{-\eta}$, we compute
\bqs
\left| \frac{1}{2 \, \pi \, \mathbf{i}} \, \int_{\Gamma_{-\eta}^{in}} \, {\rm e}^{\, n \, \tau} \, \mathbf{G}_\tau(j,j_0) \, \md \tau \right| 
\, \leq \, C \, {\rm e}^{- \, n \, \eta} \, \int_{\Gamma_{-\eta}^{in}} \, {\rm e}^{\, |j-j_0| \, \Re(\varpi(\tau))} \, |\md \tau| \, .
\eqs
Next, for all $\tau \in \Gamma^{in}_{-\eta}$ we have
\begin{align*}
\Re(\varpi(\tau)) & \leq - \, \dfrac{1}{\alpha} \, \Re(\tau) - \underbrace{\dfrac{\beta_*}{\alpha^{\, 2 \, \mu+1}} \, \Im(\tau)^{2 \, \mu}}_{\le 0} 
+\dfrac{\beta^*}{\alpha^{\, 2 \, \mu+1}} \, \Re(\tau)^{2\mu} \\
& \leq \dfrac{\eta}{\alpha} \, + \dfrac{\beta^*}{\alpha^{\, 2 \, \mu+1}} \, \eta^{2\mu}\,. 
\end{align*}
As a consequence, 
\begin{align*}
- \, n \, \eta \, + \, |j-j_0| \, \Re(\varpi(\tau)) \, & \leq \, n \, \eta \, \left( 
- \, 1 \, + \, \dfrac{|j-j_0|}{n \, \alpha} \, + \, \dfrac{|j-j_0|}{n} \, \dfrac{\beta^*}{\alpha^{\, 2 \, \mu+1}} \, \eta^{2\mu-1} \right) \\ 
&  \leq \, n \, \eta \, \left( 
- \, \dfrac{1}{2} \, + \, \dfrac{\beta^*}{2 \, \alpha^{\, 2 \, \mu}} \, \eta^{2\mu-1} \right) \, \leq \, - \, n \, \frac{\eta}{4} \, ,
\end{align*}
provided that $\eta$ is chosen small enough (the choice only depends on $\beta^*$ and $\alpha$).
\end{proof}

\subsection{Short range}

Throughout this section, we assume that $n\geq 1$ and $n \, \frac{\alpha}{2} \leq j-j_0 \leq n \, r$. Following \cite{ZH98,godillon,CF1}, we introduce a 
family of parametrized curves given by
\bqq
\Gamma_p \, := \, \left\{ \, \Re(\tau) \, - \, \dfrac{\beta^*}{\alpha^{\, 2 \, \mu}} \,  \Re(\tau)^{2 \, \mu} \, + \, \dfrac{\beta_*}{\alpha^{\, 2 \, \mu}} \, \Im(\tau)^{2 \, \mu} 
\, = \, \Psi (\tau_p)~|~ -\eta \leq \Re(\tau)\leq \tau_p \, \right\}
\label{contourGp}
\eqq
with $\Psi(\tau_p) \, := \, \tau_p - \, \frac{\beta^*}{\alpha^{\, 2 \, \mu}} \, \tau_p^{\, 2 \, \mu}$. Note that these curves intersect the real axis at $\tau_p$. 
We also let
\bqs
\zeta \, := \, \dfrac{j-j_0 \, - \, n \, \alpha}{2 \, \mu \, n} \, ,\quad \text{ and } \quad \gamma \, := \, \dfrac{j-j_0}{n} \, \dfrac{\beta^*}{\alpha^{\, 2 \, \mu}} \, > \, 0 \, ,
\eqs
and  define $\rho\left(\frac{\zeta}{\gamma}\right)$ as the unique real root to the equation
\bqs
- \, \zeta \, + \, \gamma \, x^{2 \, \mu-1} \, = \, 0 \, ,
\eqs
that is 
\bqs
\rho \left(\frac{\zeta}{\gamma}\right) \, := \, \mathrm{sgn}(\zeta) \, \left( \frac{|\zeta|}{\gamma} \right)^{\frac{1}{2 \, \mu-1}} \, .
\eqs
The specific  value of $\tau_p$ is now fixed depending on the ratio $\frac{\zeta}{\gamma}$ as follows
\bqs
\tau_p \, := \, \left\{
\begin{split}
-\dfrac{\eta}{2} &\quad \text{ if } \quad \rho\left(\frac{\zeta}{\gamma}\right) <- \frac{\eta}{2} \, ,\\
\rho\left(\frac{\zeta}{\gamma}\right) &\quad \text{ if } \quad  -\frac{\eta}{2}\leq \rho\left(\frac{\zeta}{\gamma}\right) \leq \varepsilon_0\,,\\
\varepsilon_0 &\quad \text{ if } \quad \rho\left(\frac{\zeta}{\gamma}\right) > \varepsilon_0 \, , 
\end{split}
\right.
\eqs
where $0<\varepsilon_0<\varepsilon$ is chosen such that $\Gamma_p$ with $\tau_p=\varepsilon_0$ intersects the segment $\left\{-\eta+\mathbf{i} \, \ell 
~|~ \ell \in[-\pi,\pi]\right\}$ precisely on the boundary\footnote{This is possible because the curves $\Gamma_p$ are symmetric with respect to the real 
axis.} of $B_\varepsilon(0)$. Finally, let us note that as $n \, \frac{\alpha}{2} \leq j-j_0 \leq n \, r$, we have $-\frac{\alpha}{4 \, \mu} \leq \zeta \leq 
\frac{r-\alpha}{2 \, \mu}$. As $r > \alpha =\lambda \, a$ (see Lemma \ref{lem:Bernstein}), the region where $-\frac{\eta}{2} \leq \rho \left( \frac{\zeta}{\gamma} 
\right) \leq \varepsilon_0$ holds is not empty. From now on, we will treat each subcase separately.

\begin{figure}[t!]
  \centering
  \includegraphics[width=.45\textwidth]{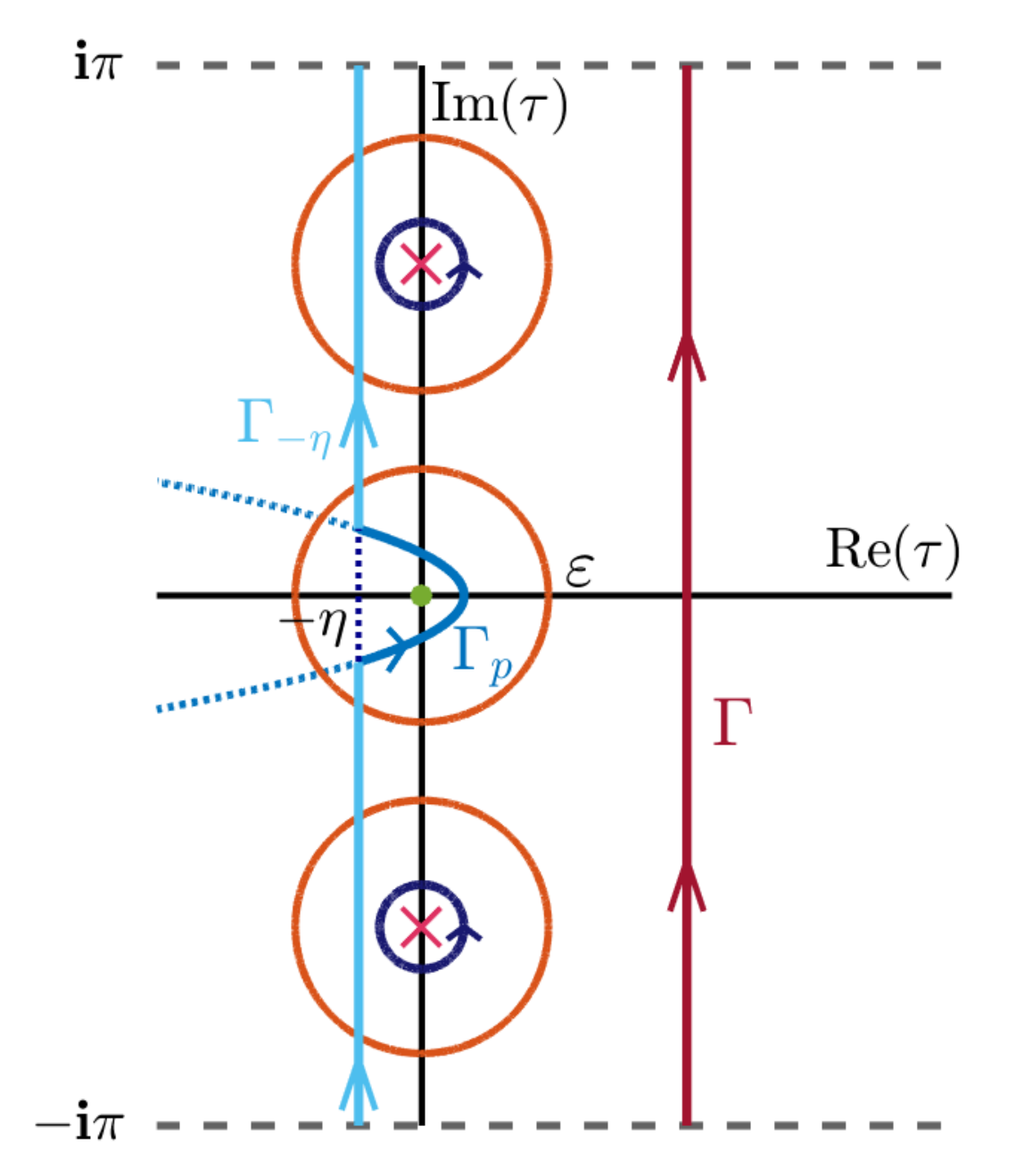}
  \caption{Illustration of the contour used in the case $-\frac{\eta}{2} \leq \rho \left( \frac{\zeta}{\gamma} \right) \leq \varepsilon_0$ when $n \, \frac{\alpha}{2} 
  \leq j-j_0 \leq n \, r$. We deform the initial contour $\Gamma$ (dark red) into the contour which consists of the parametrize curve $\Gamma_p$ near the 
  origin (blue) and a portion of the segment $\Gamma_{-\eta} = \left\{ -\eta +\mathbf{i} \, \ell ~|~ \ell \in[-\pi,\pi] \right\}$ (light blue). Note that $\Gamma_{-\eta} 
  \cap B_\varepsilon (0) \neq \emptyset$. Near each pole on the imaginary axis (pink cross), we use the residue theorem which is symbolized by the small 
  oriented circles (dark blue) surrounding each of them.}
  \label{fig:contour}
\end{figure}

\begin{lemma}
\label{lem11}
There exist constants $C>0$ and $M>0$ such that for $n\geq 1$ and $-\frac{\eta}{2} \leq \rho \left( \frac{\zeta}{\gamma} \right) \leq \varepsilon_0$, the following 
estimate holds:
\bqs
\left| \, \G^{\, n}(j,j_0) \, - \, \sum_{k=1}^K \, \underline{{\bf w}}_k(j,j_0) \, {\rm e}^{\, n \, \mathbf{i} \, \underline{\theta}_k} \, \right| \, \leq \, 
\dfrac{C}{n^{\frac{1}{2 \, \mu}}} \, \exp \left( - \, \dfrac{1}{M} \, \left( \dfrac{|j -j_0 -\alpha \, n|}{n^{\frac{1}{2 \, \mu}}} \right)^{\frac{2 \, \mu}{2 \, \mu-1}} \right) \, .
\eqs
\end{lemma}

\begin{proof}
We will consider a contour depicted in Figure~\ref{fig:contour} which consists of the parametrized curve $\Gamma_p$ near the origin and otherwise is the 
segment $\Gamma_{-\eta}$. We will denote  $\Gamma_{-\eta}^{in}$ and $\Gamma_{-\eta}^{out}$, the portions of the segment $\Gamma_{-\eta}$ which lie 
either inside or outside $B_\varepsilon(0)$ with $|\Im(\tau)|\leq \pi$. Using the residue theorem, we have that
\bqs
\G^{\, n}(j,j_0) \, = \, \sum_{k=1}^K \, \underline{{\bf w}}_k(j,j_0) \, {\rm e}^{\, n \, \mathbf{i} \, \underline{\theta}_k} \, + \, 
\dfrac{1}{2 \, \pi \, \mathbf{i}} \, \int_{\Gamma_{-\eta}^{in} \cup \Gamma_{-\eta}^{out}} \, {\rm e}^{\, n \, \tau} \, \mathbf{G}_\tau(j,j_0) \, \md \tau 
\, + \, \dfrac{1}{2 \, \pi \, \mathbf{i}} \int_{\Gamma_p} \, {\rm e}^{\, n \, \tau} \, \mathbf{G}_\tau(j,j_0) \, \md \tau \, .
\eqs
Computations along $\Gamma_{-\eta}^{out}$ are similar to the previous cases, and we directly get
\bqs
\left| \, \dfrac{1}{2 \, \pi \, \mathbf{i}} \, \int_{\Gamma_{-\eta}^{out}} \, {\rm e}^{\, n \, \tau} \, \mathbf{G}_\tau(j,j_0) \, \md \tau \, \right| \, \leq \, 
C \, {\rm e}^{- \, n \, \eta \, - \, c \, |j-j_0|} \, .
\eqs
For all $\tau \in \Gamma_{-\eta}^{in}$, we use that $\Im(\tau)^2 \geq \Im(\tau_*)^2$ where $\tau_*=-\eta+\mathbf{i} \, \ell_*$ and $\ell_*>0$ is the positive 
root of
\bqs
- \, \eta \, - \dfrac{\beta^*}{\alpha^{\, 2 \, \mu}} \,  \eta^{\, 2 \, \mu} \, +\, \dfrac{\beta_*}{\alpha^{\, 2 \, \mu}} \, \ell_*^{\, 2 \, \mu} = \, \Psi (\tau_p) \, .
\eqs
That is, the point $\tau_*=-\eta+\mathbf{i} \, \ell_*$ lies at the intersection of $\Gamma_p$ and  the segment $\left\{ -\eta +\mathbf{i} \, \ell ~|~ \ell \in[-\pi,\pi] 
\right\}$ with $\tau_*\in B_\varepsilon(0)$. As a consequence, for all $\tau \in \Gamma_{-\eta}^{in}$ we have
\begin{align*}
\Re(\varpi(\tau)) \, & \leq \, \dfrac{\eta}{\alpha} \, + \, \dfrac{\beta^*}{\alpha^{\, 2 \, \mu+1}} \, \eta^{\, 2 \, \mu} 
\, - \, \dfrac{\beta_*}{\alpha^{\, 2 \, \mu+1}} \, \Im(\tau)^{\, 2 \, \mu} \\
&= \, - \, \dfrac{\tau_p}{\alpha} \,+ \, \dfrac{\beta^*}{\alpha^{\, 2 \, \mu+1}} \, \tau_p^{2 \, \mu} -\dfrac{\beta_*}{\alpha^{\, 2 \, \mu+1}} \, 
\underbrace{\left( \, \Im(\tau)^{2 \, \mu} \, - \, \ell_*^{2 \, \mu} \, \right)}_{\geq 0} \\
&\leq \, - \, \dfrac{\tau_p}{\alpha} \, + \, \dfrac{\beta^*}{\alpha^{\, 2 \, \mu+1}} \, \tau_p^{2 \, \mu} \, .
\end{align*}
Thus, we have
\begin{align*}
n \, \Re(\tau) \, + \, (j-j_0) \, \Re(\varpi(\tau)) & \leq \, - \, n \, \eta +(j-j_0) \, \left( - \, \dfrac{\tau_p}{\alpha} \, + \, 
\dfrac{\beta^*}{\alpha^{\, 2 \, \mu+1}} \, \tau_p^{\, 2 \, \mu} \,  \right) \\
&= \, \dfrac{n}{\alpha} \, \left[ \, - \, \eta \, \alpha \, + \, \dfrac{(j-j_0)}{n} \, \left( - \, \tau_p \, + \, \dfrac{\beta^*}{\alpha^{\, 2 \, \mu}} \, \tau_p^{\, 2 \, \mu} \right) \, \right] \\
&= \, \dfrac{n}{\alpha} \, \left[ \, - \, (\eta \, + \, \tau_p ) \, \alpha \, - \, 2 \, \mu \, \zeta \, \tau_p \, + \, \gamma \, \tau_p^{\, 2 \, \mu} \, \right] \\
&= \, \dfrac{n}{\alpha}  \, \left[ \, - \, (\eta \, + \, \tau_p ) \, \alpha \, + \, (1-2 \, \mu) \, \gamma \, \left( 
\dfrac{|\zeta|}{\gamma} \right)^{\frac{2 \, \mu}{2 \, \mu-1}} \, \right] \, ,
\end{align*}
for all $\tau\in\Gamma_{-\eta}^{in}$. Finally,  as $-\frac{\eta}{2} \leq \rho(\frac{\zeta}{\gamma})=\tau_p$ we have $\eta +\tau_p\geq \frac{\eta}{2}$, 
and we obtain an estimate of the form
\bqs
\left| \, \dfrac{1}{2 \, \pi \, \mathbf{i}} \, \int_{\Gamma_{-\eta}^{in}} \, {\rm e}^{n \, \tau} \, \mathbf{G}_\tau(j,j_0) \, \md \tau \right| \, \leq \, 
C \, {\rm e}^{- \, n \, \frac{\eta}{2} -\frac{n}{\alpha} \, (2 \, \mu-1) \, \gamma \, \left(\frac{|\zeta|}{\gamma}\right)^{\frac{2 \, \mu}{2 \, \mu-1}}} 
\, \leq \, C \, {\rm e}^{- \, n \, \frac{\eta}{2} \, - \, c \, n \, |\zeta|^{\frac{2 \, \mu}{2 \, \mu-1}}} \, ,
\eqs
since $\gamma$ is bounded from below and from above by positive constants.

We now turn our attention to the integral along $\Gamma_p$. We first notice that for all  $\tau \in \Gamma_p \subset B_\varepsilon(0)$, we have
\bqs
\Re(\tau) \, \leq \, \tau_p \, - \, c_* \, \Im(\tau)^{2 \, \mu} \, ,
\eqs
for some constant $c_*>0$. As a consequence, we obtain the upper bound
\begin{align*}
n \, \Re(\tau)+(j-j_0) \, \Re(\varpi(\tau)) & \leq n \, \Re(\tau) \, - \, \dfrac{j-j_0}{\alpha} \, \Re(\tau) \, + \, \dfrac{\beta^*(j-j_0)}{\alpha^{\, 2 \, \mu+1}} \, \Re(\tau)^{2 \, \mu} 
\, -  \, \dfrac{\beta_*(j-j_0)}{\alpha^{\, 2 \, \mu+1}} \, \Im(\tau)^{2 \, \mu} \\
& \leq \, n \, (\Re(\tau)-\tau_p) \, + \, \dfrac{n}{\alpha} \, \Big[ \, - \, 2 \, \mu \, \zeta \, \tau_p \,+ \, \gamma \, \tau_p^{\, 2 \, \mu} \, \Big] \\
& \leq \, - \, n \, c_* \, \Im(\tau)^{2 \, \mu} \, - \, \frac{n}{\alpha} \, (2 \, \mu-1) \, \gamma \, \left( \frac{|\zeta|}{\gamma} \right)^{\frac{2 \, \mu}{2 \, \mu-1}} \, ,
\end{align*}
for all  $\tau \in \Gamma_p \subset B_\varepsilon(0)$. As a consequence, we can derive the following bound
\begin{align*}
\left| \, \dfrac{1}{2 \, \pi \, \mathbf{i}} \, \int_{\Gamma_p} \, {\rm e}^{n \, \tau} \, \mathbf{G}_\tau(j,j_0) \, \md \tau \, \right| & \leq \, 
C \, \int_{\Gamma_p} {\rm e}^{\, n \, \Re(\tau) \, + \, (j-j_0) \, \Re(\varpi(\tau))} \, |\md \tau| \\
&\leq \, C \, {\rm e}^{- \, \frac{n}{\alpha} \,(2 \, \mu-1) \, \gamma \, \left( \frac{|\zeta|}{\gamma} \right)^{\frac{2 \, \mu}{2 \, \mu-1}}} 
\, \int_{\Gamma_p} \, {\rm e}^{- \, n \, \frac{c_*}{2} \, \Im(\tau)^{2 \, \mu}} \, |\md \tau| \\
&\leq \, C \, \dfrac{{\rm e}^{- \, \frac{n}{\alpha} \, (2 \, \mu-1) \, \gamma \, \left(\frac{|\zeta|}{\gamma}\right)^{\frac{2 \, \mu}{2 \, \mu-1}}}}{n^{\frac{1}{2 \, \mu}}} 
\, \leq \, C \, \dfrac{{\rm e}^{- \, c \, n \, |\zeta|^{\frac{2 \, \mu}{2 \, \mu-1}}}}{n^{\frac{1}{2 \, \mu}}} \, ,
\end{align*}
where we use again that $\gamma$ is bounded from below and from above by positive constants. At the end of the day, we see that the leading 
contribution is the one coming from the integral along $\Gamma_p$.
\end{proof}

\noindent Finally, we treat the last two cases altogether.

\begin{lemma}
\label{lem12}
There exist constants $C>0$ and $c_\star>0$ such that for $n \geq 1$ and $ \rho \left( \frac{\zeta}{\gamma} \right) < -\frac{\eta}{2}$ or $\rho \left( 
\frac{\zeta}{\gamma} \right) > \varepsilon_0$ there holds:
\bqs
\left| \, \G^{\, n}(j,j_0) \, - \, \sum_{k=1}^K \, \underline{{\bf w}}_k(j,j_0) \, {\rm e}^{\, n \, \mathbf{i} \, \underline{\theta}_k} \, \right| \, \leq \, 
C \, {\rm e}^{- \, n \, c_\star} \, .
\eqs
\end{lemma}

\begin{proof}
We only present the proof in case $\rho(\zeta/\gamma) > \epsilon_0$ as the proof for $\rho \left( \frac{\zeta}{\gamma} \right) < -\frac{\eta}{2}$ follows 
similar lines. We deform the contour $\Gamma$ into $\Gamma_p \cup \Gamma_{-\eta}^{out}$ where $\Gamma_{-\eta}^{out}$ are the portions of 
$\Gamma_{-\eta}$ which lie outside $B_\varepsilon(0)$ with $|\Im(\tau)|\leq \pi$. We recall that we choose $\tau_p = \varepsilon_0$ here, so the 
curve $\Gamma_p$ intersects $\partial B_\varepsilon(0)$ precisely at $\Re(\tau)=-\eta$. In that case, we have that for all $\tau \in \Gamma_p$
$$
n \, \Re(\tau) \, + \, (j-j_0) \, \Re(\varpi(\tau)) \, \leq \, - \, n \, c_* \, \Im(\tau)^{2 \, \mu} \, + \, 
\dfrac{n}{\alpha} \, \left( \, - \, 2 \, \mu \, \zeta \, \varepsilon_0 \, + \, \gamma \, \varepsilon_0^{\, 2 \, \mu} \,  \right) \, .
$$
But as $\rho(\zeta/\gamma) > \varepsilon_0$ we get that $\zeta>0$ and $\zeta > \varepsilon_0^{2 \, \mu-1} \, \gamma$, the last term in the previous 
inequality is  estimated via
\bqs
- \, 2 \, \mu \, \zeta \, \varepsilon_0 \, + \, \gamma \, \varepsilon_0^{2 \, \mu} \,  < \, 
(1 \, - \, 2 \, \mu) \, \gamma \, \varepsilon_0^{2 \, \mu} \, .
\eqs
As a consequence, we can derive the following bound
\bqs
\left| \, \dfrac{1}{2 \, \pi \, \mathbf{i}} \, \int_{\Gamma_p} \, {\rm e}^{n \, \tau} \, \mathbf{G}_\tau(j,j_0) \, \md \tau \, \right| \, \leq \, C \, 
\dfrac{{\rm e}^{- \, \frac{n}{\alpha} \, (2 \, \mu-1) \, \gamma \, \varepsilon_0^{2 \, \mu}}}{n^{\frac{1}{2 \, \mu}}} \, .
\eqs
With our careful choice of $\varepsilon_0>0$, the remaining contribution along segments $\Gamma_{-\eta}^{out}$ can be estimated as usual as 
\bqs
\left| \, \dfrac{1}{2 \, \pi \, \mathbf{i}} \, \int_{\Gamma_{-\eta}^{out}} \, {\rm e}^{n \, \tau} \, \mathbf{G}_\tau(j,j_0) \, \md \tau \, \right| \, \leq \, 
C \, {\rm e}^{- \, n \, \eta \, - \, c \, (j-j_0)} \, ,
\eqs
as $|\tau| \geq \varepsilon$ for $\tau \in \Gamma_{-\eta}^{out}$. The conclusion of Lemma \ref{lem12} follows.
\end{proof}

We can now combine Lemma~\ref{lem9}, Lemma~\ref{lem10}, Lemma~\ref{lem11} and Lemma~\ref{lem12} to obtain the result of Proposition \ref{prop1}. 
Indeed, we observe that in Lemma~\ref{lem9}, Lemma~\ref{lem10} and Lemma~\ref{lem12}, the obtained exponential bounds can always be subsumed 
into Gaussian-like estimates. (Lemma~\ref{lem11} yields the worst estimate of all.) For instance, in Lemma~\ref{lem10}, the considered integers $j,j_0,n$ 
satisfy $1 \leq j-j_0 \leq n \, \alpha/2$, which implies 
$$
- \, n \, \le \, \dfrac{j \, - \, j_0 \, - \, n \, \alpha}{\alpha} \,  \le - \, \dfrac{n}{2} \, ,
$$
and therefore:
$$
- \, n \, \le \, - \, \omega \, \left( \dfrac{|j -j_0 -n \, \alpha|}{n^{\frac{1}{2 \, \mu}}} \right)^{\frac{2 \, \mu}{2 \, \mu-1}} \, ,
$$
for some sufficiently small constant $\omega > 0$. It remains to explain why Proposition \ref{prop1} implies Theorem \ref{thm1}.

\subsection{Proof of the main result}

We let $w \in \mathcal{H}$, and first remark that for any integer $n$, the sequence $\T^n \, w \in \mathcal{H}$ is given by:
\bqs
\forall \, j \geq 1\, , \quad (\T^n \, w)_j \, = \, \sum_{j_0\geq 1}\, \G^{\, n}(j,j_0) \, w_{j_0} \, .
\eqs
From Proposition~\ref{prop1}, we can decompose $\G^n(j,j_0)$ into two pieces
\bqs
\mathcal{G}^{\, n}(j,j_0) \, = \, \sum_{k=1}^K \, \underline{{\bf w}}_k(j,j_0) \, {\rm e}^{\, n \, \mathbf{i} \, \underline{\theta}_k} \, + \, 
\tilde{\mathcal{G}}^{\, n}(j,j_0) \, ,
\eqs
where the remainder term $\tilde{\mathcal{G}}^{\, n}(j,j_0)$ satisfies the generalized Gaussian estimate of Proposition \ref{prop1}. From the exponential 
bound \eqref{bornesprofils} and Proposition \ref{prop1}, we have:
\bqs
\sum_{j_0 \geq 1} \, \Big| \, \G^{\, n}(j,j_0) \, w_{j_0} \, \Big| \, \leq \, C \, {\rm e}^{- \, c \, j} \, \sum_{j_0 \geq 1} \, {\rm e}^{- \, c \, j_0} \, |w_{j_0}| \, + \, 
\dfrac{C}{n^{\frac{1}{2 \, \mu}}} \, \sum_{j_0 \geq 1} \, 
\exp \left( -\, \omega \, \left( \dfrac{|j-j_0- \alpha \, n|}{n^{\frac{1}{2 \, \mu}}} \right)^{\frac{2 \, \mu}{2 \, \mu-1}}\right) \, |w_{j_0}| \, .
\eqs
Noting that the sequence $\left({\rm e}^{- \, c \, j} \right)_{j\geq1}$ is in $\ell^2$, we get that
\bqs
\sum_{j \geq 1} \, \left( {\rm e}^{- \, c \, j} \, \sum_{j_0\geq 1} \, {\rm e}^{- \, c \, j_0} \, |w_{j_0}| \right)^2 \, \le C \, \| \, w \, \|_{\mathcal{H}}^2 \, .
\eqs
Now for the second term, we observe that the sequence defined as
\bqs
\forall \, j \in \mathbb{Z} \, ,\quad \mathbf{g}_j \, := \, \dfrac{1}{n^{\frac{1}{2 \, \mu}}} \, 
\exp \left( - \, \omega \, \left( \dfrac{|j- \alpha \, n|}{n^{\frac{1}{2 \, \mu}}} \right)^{\frac{2 \, \mu}{2 \, \mu-1}} \right) \, ,
\eqs
is bounded (with respect to $n \in \N^*$) in $\ell^1(\Z)$. Using the Young's convolution inequality $\ell^1(\Z) \star \ell^2(\Z) \rightarrow \ell^2(\Z)$, 
we thus obtain the uniform in time bound:
\bqs
\sum_{j\geq 1} \, \left( \, \dfrac{1}{n^{\frac{1}{2 \, \mu}}} \, \sum_{j_0 \geq 1} \, 
\exp \left( -\, \omega \, \left( \dfrac{|j-j_0- \alpha \, n|}{n^{\frac{1}{2 \, \mu}}} \right)^{\frac{2 \, \mu}{2 \, \mu-1}}\right) \, |w_{j_0}| \, \right)^2 
\, \le \, C \, \| \, w \, \|_{\mathcal{H}}^2 \, .
\eqs
This completes the proof that our operator $\T$ is power bounded on $\mathcal{H}$.

\section{An illustrative example}
\label{sect:example}

\begin{figure}[t!]
  \centering
  \includegraphics[width=.45\textwidth]{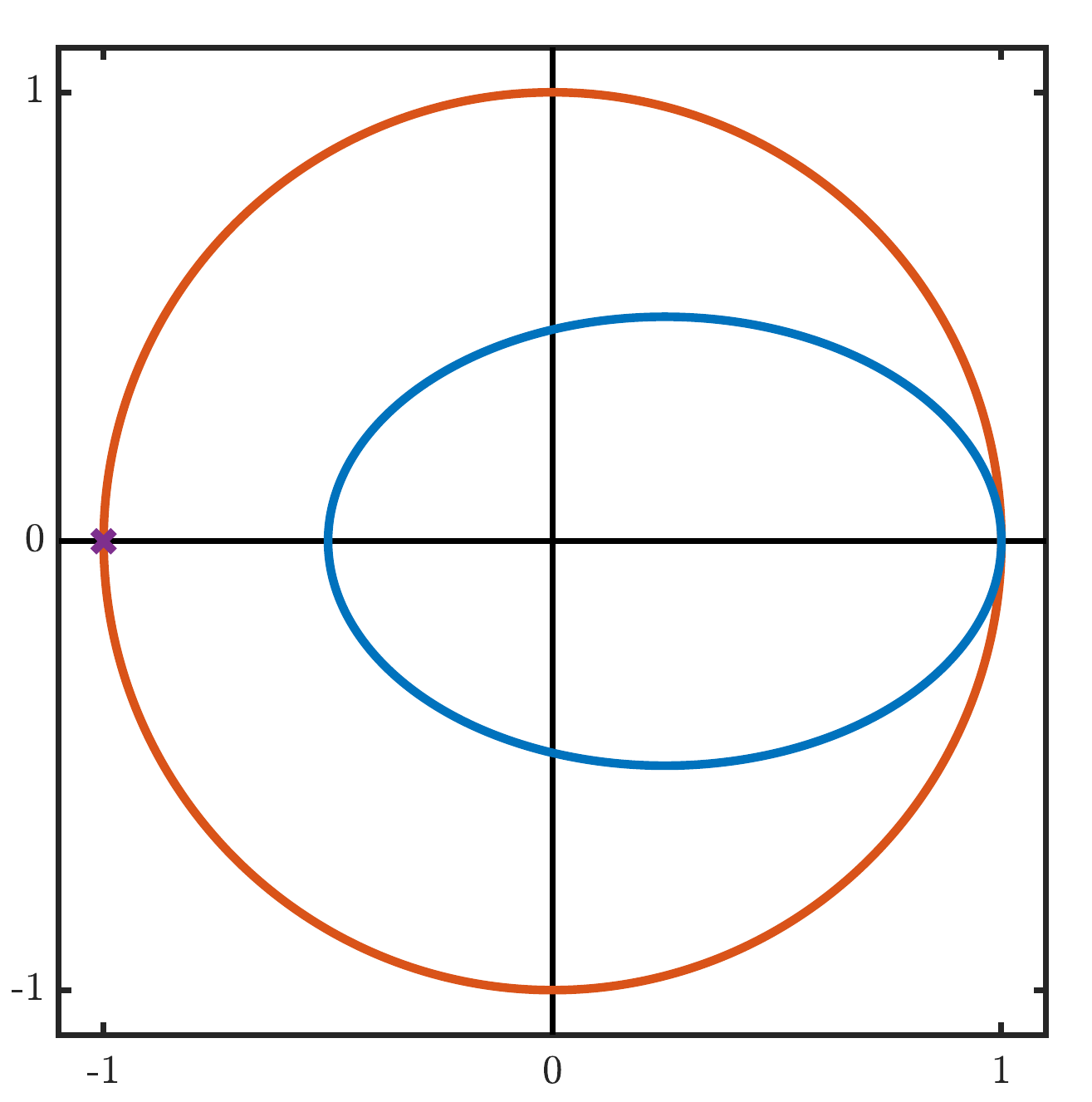}
  \caption{Spectrum of the modified Lax-Friedrichs numerical scheme \eqref{LFnum}-\eqref{LFnumBC} with $\lambda a=1/2$, $D=3/4$ and 
  $b=-1-\frac{2\sqrt{5}}{5}$. The blue curve is the essential spectrum of $\T$ and we have indicated by a cross the presence of eigenvalue at 
  $\underline{z}=-1$. Our Assumptions~\ref{hyp:1}-\ref{hyp:2} are satisfied in that case.}
  \label{fig:LFspectrm}
\end{figure}

We illustrate our main result by considering the modified Lax-Friedrichs numerical scheme which reads
\begin{equation}
\label{LFnum}
u_j^{n+1} \, = \, u_j^n \, + \, \dfrac{D}{2} \, \left(u_{j-1}^n \, - \, 2 \, u_j^n \, + \, u_{j+1}^n \right) \, - \, 
\dfrac{\lambda \, a}{2} \, \left( u_{j+1}^n \, - \, u_{j-1}^n \right) \, ,\quad j \ge 1 \, ,
\end{equation}
where $D>0$ and $\lambda a>0$, along with some specific boundary condition at $j=0$ which we shall specify later. Using our formalism from 
\eqref{schema-int}, we have $p=r=1$ and
\bqs
a_{-1} \, = \, \dfrac{D \, + \, \lambda \, a}{2} \, ,\quad a_0 \, = \, 1 \, - \, D \, ,\quad \text{ and } \quad a_1 \, = \, \dfrac{D \, - \, \lambda \, a}{2} \, .
\eqs
We readily note that our consistency conditions \eqref{hyp:consistance} are satisfied. Next, if we denote 
\bqs
F(\theta) \, := \, \sum_{\ell=-1}^1 \, a_\ell \, \mathrm{e}^{\, \mathbf{i} \, \theta \, \ell} \, ,\quad \theta \in [-\pi,\pi] \, ,
\eqs
then we have
\bqs
F(\theta) \, = \, 1 \, - \, D \, + \, D \, \cos(\theta) \, - \, \mathbf{i} \, \lambda \, a \, \sin(\theta) \, .
\eqs
As a consequence, provided that $0 < \lambda \, a < 1$ and $(\lambda \, a)^2 < D < 1$, we get 
\bqs
\forall \, \theta \in [-\pi,\pi] \setminus \{ 0 \} \, ,\quad \left| \, F(\theta) \, \right| \, < \, 1 \, ,
\eqs
such that the dissipativity condition \eqref{hyp:stabilite1} is also verified. Next, we compute that
\bqs
F(\theta) \, = \, 1 \, - \, \mathbf{i} \, \lambda \, a \, \theta \, - \, \dfrac{D}{2} \, \theta^2 \, + \, O(\theta^3) \, , 
\eqs
as $\theta$ tends to $0$. We thus deduce that \eqref{hyp:stabilite2} is satisfied with
\bqs
\mu \, := \, 1 \, ,\quad \text{ and } \quad \beta \, := \, \dfrac{D \, - \, (\lambda \, a)^2}{2} \, > \, 0 \, .
\eqs
Assumption \ref{hyp:1} is thus satisfied provided that we have $0 < \lambda \, a < 1$ and $(\lambda \, a)^2 < D < 1$. We also assume from now on 
$D \neq \lambda \, a$ so that the coefficient $a_1$ is nonzero.

We now prescribe a boundary condition for \eqref{LFnum} which will ensure that our Assumption~\ref{hyp:2} on the Lopatinskii determinant is satisfied. 
That  is, we want to find $\underline{z}\in\mathbb{S}^1\setminus\{1\}$ which is an eigenvalue for $\T$. This means that at this point $\underline{z}$ the 
boundary condition must be adjusted so as to have $\text{\rm Ker } \mathcal{B} \, \cap \, \E^s (\underline{z}) \neq \{ 0 \}$. We use a boundary 
condition of the form given in \eqref{schema-bc} with $p_b=p=1$:
\bqs
u_0^{n} \, = \, b \, u_1^{n} \, ,\quad n \geq 1 \, ,
\eqs
where $b \in \R$ is a constant. In order to ensure that $\text{\rm Ker } \mathcal{B} \, \cap \, \E^s (\underline{z}) \neq \emptyset$ is satisfied, we impose that 
\bqs
1 \, = \, b \, \kappa_s(\underline{z}) \, ,
\eqs
where $\kappa_s(\underline{z})$ refers to the (unique) stable eigenvalue of $\mathbb{M}(\underline{z})$. Finally, we select $\underline{z}=-1$. This is 
the only value on the unit circle, apart from $z=1$, which ensures that $\kappa_s(\underline{z})$ is real. Note that $\kappa_s(-1)$ has the exact expression
\bqs
\kappa_s(-1) \, = \, \dfrac{- \, 1 \, - \, a_0}{2 \, a_1} \, + \, \sqrt{ \left( \dfrac{- \, 1 \, - \, a_0}{2 \, a_1} \right)^2 \, - \, \dfrac{a_{-1}}{a_1}} \, \neq \, 0 \, .
\eqs
Our actual boundary condition is thus
\bqq
u_0^{n}=\frac{1}{\kappa_s(-1)}\, u_1^{n}, \quad n \geq 1.
\label{LFnumBC}
\eqq
With that specific choice, we easily see that $\text{\rm Ker } \mathcal{B} \, \cap \, \E^s (z)$ is nontrivial for $z \in \Ubar \setminus \{ 1 \}$ if and only if 
$z=-1$, for the Lopatinskii determinant equals $1-\kappa_s(z)/\kappa_s(-1)$, and the equation $\kappa_s(z)=\kappa_s(-1)$ has a unique solution 
given precisely by $z=\underline{z}=-1$. Moreover, $-1$ is a simple root of the Lopatinskii determinant. Hence Assumption \ref{hyp:2} is satisfied 
with the choice \eqref{LFnumBC}.

\begin{figure}[t!]
  \centering
  \includegraphics[width=.45\textwidth]{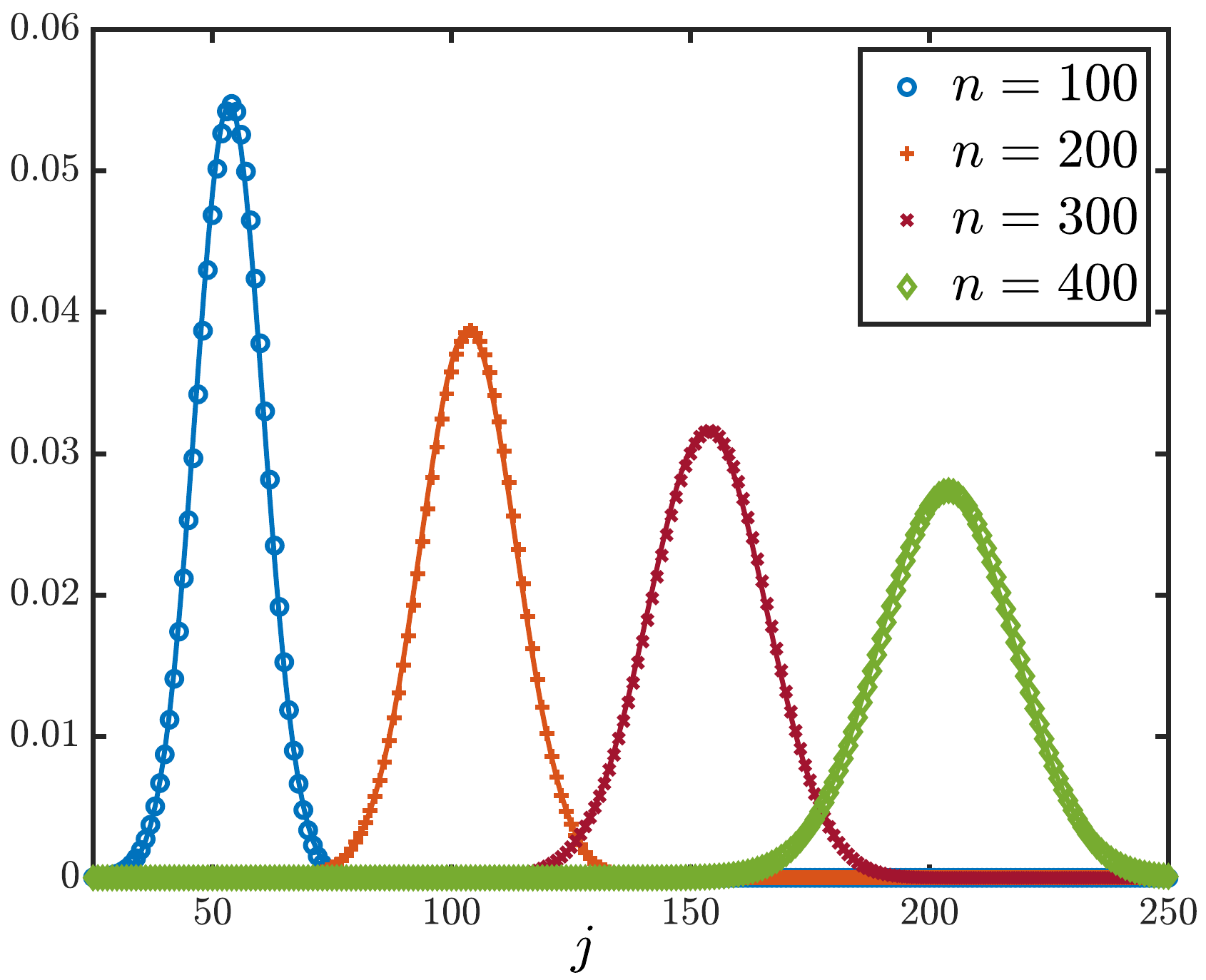}\hspace{0.75cm}
  \includegraphics[width=.465\textwidth]{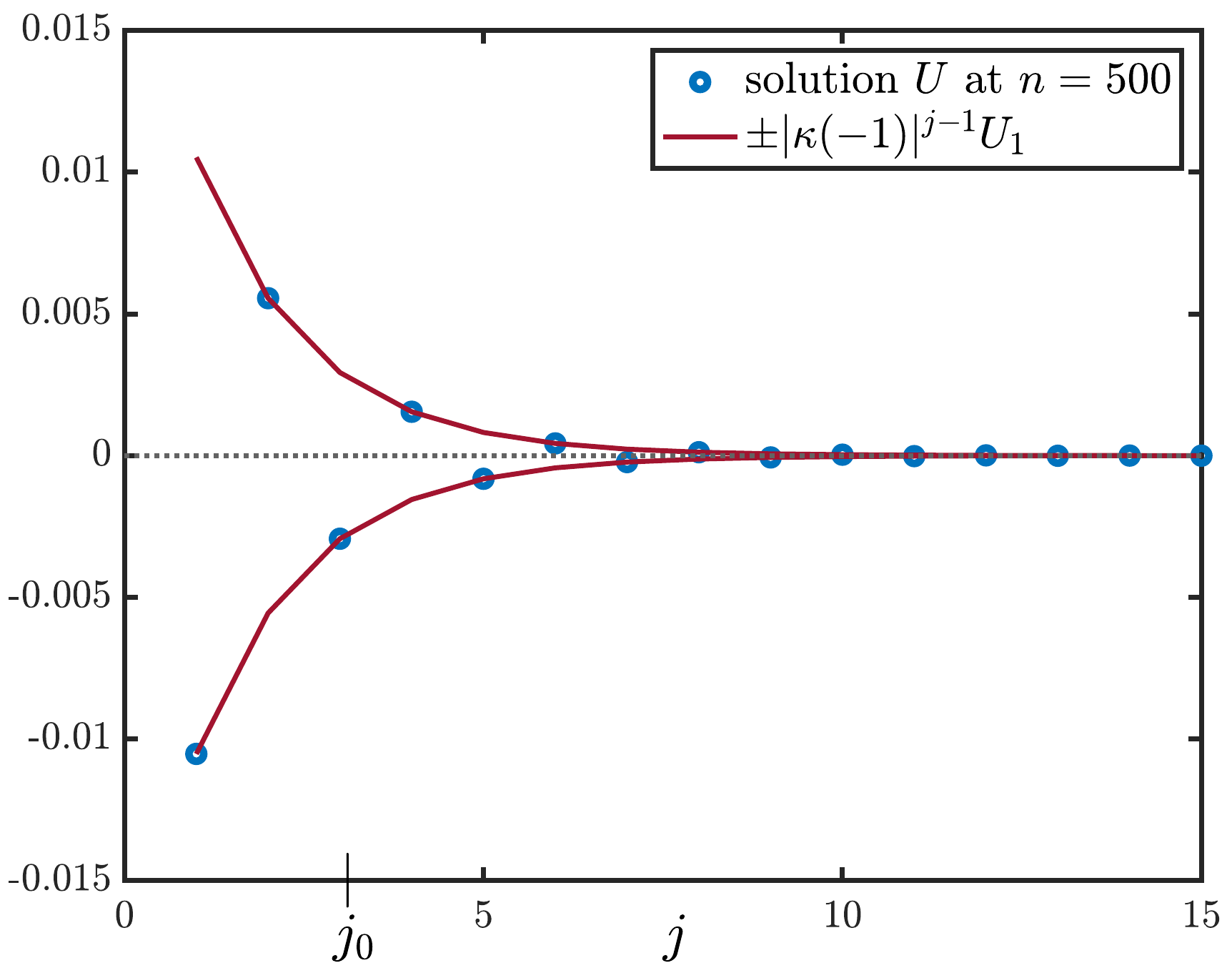}
  \caption{We illustrate our main Theorem~\ref{thm1} and Proposition~\ref{prop1} for the modified Lax-Friedrichs numerical scheme \eqref{LFnum}-\eqref{LFnumBC} 
  with $\lambda \, a=1/2$,  $D=3/4$ and $b=-1-\frac{2\sqrt{5}}{5}$. We started with an initial condition given by the Dirac mass at $j_0=3$. In the left figure, we 
  represent the Green's function at different time iterations and compare with a fixed Gaussian profile centered at $j-j_0=\lambda \, a \, n$ away from the boundary 
  $j=1$. In the right figure, we highlight the behavior of the Green's function near the boundary. We represent the solution (blue circules) after $500$ time iterations 
  and show that it corresponds to a so-called surface wave given by the eigenvalue at $\underline{z}=-1$ of $\T$.}
  \label{fig:LFnum}
\end{figure}

Note that the modified Lax-Friedrichs numerical scheme \eqref{LFnum}-\eqref{LFnumBC} is (formally) consistant with discretization of the transport 
equation 
\bqs
\left\{
\begin{split}
& \partial_t u \, + \, a \, \partial_x u \, = \, 0 \, , \quad t>0, \quad x>0 \, ,\\
& u(t,0) = 0 \, , \quad t>0 \, , \\
& u(0,x) = u_0(x) , \quad x>0 \, ,
\end{split}
\right.
\eqs
for some given (smooth) initial condition $u_0$.

We present in Figure~\ref{fig:LFspectrm} the spectrum of $\T$ associated to the modified Lax-Friedrichs numerical scheme \eqref{LFnum}-\eqref{LFnumBC} 
with $\lambda \, a=1/2$,  $D=3/4$ and $b=-1-\frac{2\sqrt{5}}{5}$. In Figure~\ref{fig:LFnum}, we illustrate the decomposition given in Proposition~\ref{prop1} 
where the temporal Green's function decomposes into two parts: a boundary layer part given by $\underline{\mathbf{w}}_1(j,j_0) \, (-1)^n$ which is exponentially 
localized in both $j$ and $j_0$ and a generalized Gaussian part which is advected away along $j-j_0=\lambda \, a \, n$. We start with an initial condition 
given by the Dirac mass at $j_0=3$. We remark that the Green's function at different time iterations compares well with a fixed Gaussian profile centered 
at $j-j_0 = \lambda \, a \, n$ away from the boundary $j=1$. We also visualize the behavior of the solution near the boundary for $1\leq j \leq 15$ and 
shows that up to a constant, depending on $j_0$, the envelope of the Green's function is given by $\pm \, |\kappa_s(-1)|^{j-1}$.

\appendix

\section{Proofs of intermediate results}
\label{sec:appendix}

This Appendix is devoted to the proof of several intermediate results, which are recalled here for the reader's convenience.

\subsection{The Bernstein type inequality}

\begin{lemma}
Under Assumption \ref{hyp:1}, there holds $\lambda \, a \, < \, r$.
\end{lemma}

\begin{proof}
We introduce the polynomial function:
$$
\forall \, z \in \C \, ,\quad P(z) \, := \, \sum_{\ell=-r}^p \, a_\ell \, z^{\, \ell \, + \, r} \, .
$$
Assumption \ref{hyp:1} implies that $P$ is a nonconstant holomorphic function on $\C$ and that the modulus of $P$ is not larger than $1$ 
on $\cercle$. By the maximum principle for holomorphic functions, $P$ maps $\D$ onto $\D$. In particular, since $P$ has real coefficients, 
$P$ achieves its maximum on $[0,1]$ at $1$, and we thus have $P'(1) \ge 0$. From \eqref{hyp:consistance}, we thus have $P'(1) = r - 
\lambda \, a \ge 0$. It remains to explain why $\lambda \, a$ can not equal $r$.

We assume from now on $\lambda \, a =r$ and explain why this leads to a contradiction. Multiplying \eqref{hyp:stabilite2} by $\exp (\mathbf{i} 
\, r \, \theta)$, we obtain:
$$
P \big( \, {\rm e}^{\, \mathbf{i} \, \theta} \, \big) \, = \, \exp \big( \, - \, \beta \, \theta^{\, 2 \, \mu} \, + \, O(\theta^{\, 2 \, \mu \, + \, 1}) \big) \, ,
$$
for $\theta$ close to $0$. By the unique continuation theorem for holomorphic functions, the latter expansion holds for either real or complex 
values of $\theta$. We thus choose $\theta = \varepsilon \, \exp (\mathbf{i} \, \pi /(2\, \mu))$ for any sufficiently small $\varepsilon > 0$. We have:
$$
P \big( \, {\rm e}^{\, \mathbf{i} \, \varepsilon \, \exp (\mathbf{i} \, \pi /(2\, \mu))} \, \big) \, = \, 
\exp \big( \, \beta \, \varepsilon^{\, 2 \, \mu} \, + \, O(\varepsilon^{\, 2 \, \mu \, + \, 1}) \big) \, ,
$$
which is a contradiction since $P$ maps $\D$ onto $\D$ and $\beta > 0$. We have thus proved $\lambda \, a < r$.
\end{proof}

\subsection{Proof of Lemma~\ref{lem:1}}

\begin{lemma}
Under Assumption \ref{hyp:1}, there exists $c_0>0$ such that, if we define the set:
$$
\mathcal{C} \, := \, \Big\{ \rho \, {\rm e}^{\, \mathbf{i} \, \varphi} \in \C \, / \, \varphi \in [ - \, \pi \, , \, \pi ] \quad \text{\rm and} \quad 
0 \, \le \, \rho \, \le \, 1 \, - \, c_0 \, \varphi^{\, 2 \, \mu} \Big\} \, ,
$$
then $\mathcal{C}$ is a compact star-shaped subset of $\Dbar$, and the curve:
$$
\left\{ \sum_{\ell=-r}^p \, a_\ell \, {\rm e}^{\, \mathbf{i} \, \ell \, \theta} \, / \, \theta \in [ - \, \pi \, , \, \pi ] \right\}
$$
is contained in $\mathcal{C}$.
\end{lemma}

\begin{proof}
We first choose the constant $c_0$ such that for any sufficiently small $\theta$, the point:
$$
\sum_{\ell=-r}^p \, a_\ell \, {\rm e}^{\, \mathbf{i} \, \ell \, \theta} 
$$
lies in $\mathcal{C}$. To do so, we use \eqref{hyp:stabilite2} from Assumption \ref{hyp:1} and thus write for any sufficiently small $\theta$:
$$
\sum_{\ell=-r}^p \, a_\ell \, {\rm e}^{\, \mathbf{i} \, \ell \, \theta} \, = \, \rho(\theta) \, {\rm e}^{\, \mathbf{i} \, \varphi(\theta)} \, ,
$$
with:
$$
0 \,  \le \, \rho(\theta) \, \le \, 1 \, - \, \dfrac{\beta}{2} \, \theta^{\, 2 \, \mu} \, ,\quad \text{\rm and } \quad 
\dfrac{\lambda \, a}{2} \, |\theta| \, \le \, | \, \varphi(\theta) \, | \, \le \, \dfrac{3 \, \lambda \, a}{2} \, |\theta| \, .
$$
Hence there exists $c_0>0$ and $\theta_0>0$ small enough such that, for $|\theta| \le \theta_0$, there holds:
$$
0 \,  \le \, \rho(\theta) \, \le \, 1 \, - \, c_0 \, \varphi(\theta)^{\, 2 \, \mu} \, .
$$

Let us now examine the case $\theta_0 \le |\theta| \le \pi$. By continuity and compactness, \eqref{hyp:stabilite1} yields:
$$
\sup_{\theta_0 \le |\theta| \le \pi} \, \left| \, \sum_{\ell=-r}^p \, a_\ell \, {\rm e}^{\, \mathbf{i} \, \ell \, \theta} \, \right| \, = \, 1 \, - \, \delta_0 \, ,
$$
for some $\delta_0>0$. Up to choosing $c_0$ smaller, we can always assume $c_0 \, \pi^{\, 2 \, \mu} \le \delta_0$, so for any angle $\theta$ with 
$\theta_0 \le |\theta| \le \pi$, the point:
$$
\sum_{\ell=-r}^p \, a_\ell \, {\rm e}^{\, \mathbf{i} \, \ell \, \theta} 
$$
lies in $\mathcal{C}$. The proof is thus complete.
\end{proof}

\subsection{Proof of Lemma~\ref{lem:2} on the spectral splitting}

\begin{lemma}
Under Assumption \ref{hyp:1}, let $z \in \C$ and let the matrix $\M(z)$ be defined as in \eqref{defM}. Let the set $\mathcal{C}$ be defined by Lemma 
\ref{lem:1}. Then for $z \not \in \mathcal{C}$, $\M(z)$ has:
\begin{itemize}
 \item no eigenvalue on $\cercle$,
 \item $r$ eigenvalues in $\D \setminus \{ 0 \}$,
 \item $p$ eigenvalues in $\U$ (eigenvalues are counted with multiplicity).
\end{itemize}
Furthermore, $\M(1)$ has $1$ as a simple eigenvalue, it has $r-1$ eigenvalues in $\D$ and $p$ eigenvalues in $\U$.
\end{lemma}

\begin{proof}
We are first going to show that for $z \not \in \mathcal{C}$, $\M(z)$ has no eigenvalue on the unit circle $\cercle$ (this is a classical observation that 
dates back to \cite{kreiss1}). From the definition \eqref{defM}, we first observe that for any $z \in \C$, $\M(z)$ is invertible (its kernel is trivial since 
$r \ge 1$ and $a_{-r} \neq 0$ so the upper right coefficient of $\M(z)$ is nonzero). Therefore, for any $z \in \C$, the eigenvalues of $\M(z)$ are those 
$\kappa \neq 0$ such that:
\begin{equation}
\label{relation-valeurs-propres}
z \, = \, \sum_{\ell=-r}^p \, a_\ell \, \kappa^\ell \, . 
\end{equation}
In particular, Lemma \ref{lem:1} shows that for $z \not \in \mathcal{C}$, $\M(z)$ cannot have an eigenvalue $\kappa$ on the unit circle for otherwise 
the right hand side of \eqref{relation-valeurs-propres} would belong to $\mathcal{C}$.

Since $\mathcal{C}$ is closed and star-shaped, its complementary is pathwise-connected hence connected. Therefore, the number of eigenvalues of 
$\M(z)$ in $\D$ is independent of $z \not \in \mathcal{C}$ (same for the number of eigenvalues in $\U$). Following \cite{kreiss1} (see also \cite{jfcnotes} 
for the complete details), this number is computed by letting $z$ tend to infinity for in that case, the eigenvalues of $\M(z)$ in $\D$ tend to zero (the 
eigenvalues in $\D$ cannot remain uniformly away from the origin for otherwise the right hand side of \eqref{relation-valeurs-propres} would remain 
bounded while the left hand side tends to infinity).

The final argument is the following. For any $z \not \in \mathcal{C}$, the eigenvalues of $\M(z)$ are those $\kappa \neq 0$ such that:
$$
\kappa^r \, = \, \dfrac{1}{z} \, \sum_{\ell=-r}^p \, a_\ell \, \kappa^{r+\ell} \, ,
$$
which is just an equivalent way of writing \eqref{relation-valeurs-propres}. Hence for $z$ large, the small eigenvalues of $\M(z)$ behave at the leading 
order like the roots of the reduced equation:
$$
\kappa^r \, = \, \dfrac{a_{-r}}{z} \, ,
$$
and there are exactly $r$ distinct roots close to $0$ of that equation. Hence $\M(z)$ has $r$ eigenvalues in $\D$ for any $z \not \in \mathcal{C}$.
\bigskip

There remains to examine the spectral situation for $z=1$. Using \eqref{relation-valeurs-propres} again, the eigenvalues of $\M(1)$ are exactly the 
roots $\kappa \neq 0$ to the equation:
\begin{equation}
\label{valeurs-propres-1}
1 \, = \, \sum_{\ell=-r}^p \, a_\ell \, \kappa^\ell \, . 
\end{equation}
Thanks to Assumption \ref{hyp:1} (see \eqref{hyp:consistance} and \eqref{hyp:stabilite1}), the only root of \eqref{valeurs-propres-1} on the unit circle 
is $\kappa=1$ and it is a simple root. This simple eigenvalue can therefore be extended holomorphically with respect to $z$ as a simple eigenvalue 
of $\M(z)$ for $z$ in a neighborhood of $1$. Differentiating \eqref{relation-valeurs-propres} with respect to $z$, we obtain the Taylor expansion:
$$
\kappa (z) \, = \, 1 \, - \, \dfrac{1}{\lambda \, a} \, (z \, - \, 1) \, + \, O ((z \, - \, 1)^2) \, ,
$$
so we necessarily have $\kappa (z) \in \D$ for $z \not \in \mathcal{C}$ close to $1$. This means that the eigenvalues of $\M(1)$ that are different 
from $1$ split as follows: $r-1$ of them belong to $\D$ and $p$ belong to $\U$ (for otherwise the spectral splitting between $\D$ and $\U$ for $z 
\not \in \mathcal{C}$ would not persist for $z$ close to $1$. The proof of Lemma \ref{lem:2} is now complete.
\end{proof}

\bibliographystyle{alpha}
\bibliography{CF}
\end{document}